\documentclass[reqno,a4paper]{amsart}
\usepackage[latin1]{inputenc}
\usepackage[english]{babel}
\usepackage{eucal,amsfonts,amssymb,amsmath,amsthm,epsfig,mathrsfs}
\usepackage{dsfont,cancel,soul}
\usepackage{color}
\usepackage{amscd,amsxtra}
\usepackage{enumerate}
\usepackage{latexsym}
\usepackage{bm}

\allowdisplaybreaks

\newcounter{ipotesi}

\Alph{ipotesi}
 \makeatletter \@addtoreset{equation}{section}

\makeatother \makeatletter

\newtheorem{thm}{Theorem}[section]
\newtheorem{hyp}[thm]{Hypotheses}{\rm}
{\rm}
\newtheorem{lemm}[thm]{Lemma}
\newtheorem{coro}[thm]{Corollary}
\newtheorem{prop}[thm]{Proposition}

\newtheorem{rmk}[thm]{Remark}{\rm}

\newcounter{parentenv}

\newcommand{\R}{{\mathbb R}}

\newcommand{\N}{{\mathbb N}}

\newcommand{\Rd}{\mathbb R^d}

\newcommand{\Rm}{\mathbb R^m}
\newcommand{\Cm}{\mathbb C^m}
\newcommand{\Om}{{\Omega}}

\newcommand{\re}{{\rm{Re}}\,}
\newcommand{\im}{{\rm{Im}}\,}

\newcommand{\T}{{\bm T}}

\newcommand{\g}{{\bm g}}

\newcommand{\f}{{\bm f}}

\newcommand{\uu}{{\bm u}}
\newcommand{\A}{\bm{\mathcal A}}

\newcommand{\vv}{{\bm v}}
\newcommand{\ww}{{\bm w}}
\newcommand{\zz}{{\bm z}}

\newcommand{\z}{{\bm z}}

\renewcommand{\tilde}[1]{\widetilde{#1}}

\begin{document}

\title[Generation results for strongly coupled elliptic operators]{Generation of semigroups associated to strongly coupled elliptic operator in $L^p(\Rd;\Rm)$}
\author[L. Angiuli, L. Lorenzi and E.M. Mangino ]{Luciana Angiuli$^*$, Luca Lorenzi, Elisabetta M. Mangino}
\address{L.A. \& E.M.M.:  Dipartimento di Matematica e Fisica ``Ennio De Giorgi'', Universit\`a del Salento, Via per Arnesano, I-73100 LECCE, Italy}
\address{L.L.: Dipartimento di Scienze Matematiche, Fisiche e Informatiche, Plesso di Mate\-matica, Universit\`a degli Studi di Parma, Parco Area delle Scienze 53/A, I-43124 PARMA, Italy}
\email{luciana.angiuli@unisalento.it}
\email{luca.lorenzi@unipr.it}
\email{elisabetta.mangino@unisalento.it}
\thanks{This article is based upon work from COST Action CA18232 MAT-DYN-NET, supported by COST (European Cooperation in Science and Technology). The authors are also members of G.N.A.M.P.A. of the Italian Istituto Nazionale di Alta Matematica (INdAM)}

\begin{abstract}
A class of vector-valued elliptic operators with unbounded coefficients, coupled up to the second-order is investigated in the Lebesgue space $L^p(\Rd;\Rm)$
with $p \in (1,\infty)$, providing  sufficient conditions for the generation of an analytic $C_0$-semigroup $\T(t)$. Under further assumptions, a characterization of the domain of the infinitesimal generator is given.
\end{abstract}

\maketitle

\section{Introduction}

In this paper we deal with second-order elliptic operators acting on smooth vector-valued functions $\uu:\Rd\to\R^m$ as follows:
\begin{equation}\label{operator}
\A\uu=\sum_{h,k=1}^d D_h(Q^{hk}D_k \uu)-V\uu=\A_0\uu-V\uu,
\end{equation}
where $Q^{hk}$, ($h,k=1, \ldots, d$) and $V$ are $m\times m$ matrix-valued functions. 
More precisely, we are interested in studying  when these operators generate strongly continuous semigroups in the framework of $L^p$-spaces with respect to the Lebesgue measure, achieving also information on the regularity of the semigroups and a description of the domain of their generators.
We emphasize that the systems associated with these operators can  be  strongly coupled, i.e., the second-order terms can be  coupled to each other, and that the matrix-valued functions $Q^{hk}$ and $V$ are allowed to be unbounded.

The study of operators with a second-order coupling  gives rise to several technical obstacles. 
The first one relies in the study of the dissipativity in $L^p$, thoroughly investigated in the monograph \cite[Section 4.3]{CM}:  one cannot expect dissipativity to hold true in $L^p$ for any $1\leq p\leq \infty$, unless $Q^{hk}=q_{hk} I$, where $q_{hk}$ is a scalar function.
This feature is connected with  the lack of a ``parabolic maximum modulus principle'' for the systems associated with  \eqref{operator} (see \cite{KM}), which prevents from using extrapolation arguments from  $L^\infty$ or from the space of bounded and continuous functions (where some results are available when a variant of the maximum principle holds, see e.g., \cite{AALT, AL, DL}) to $L^p$-spaces and conversely.
These obstructions have been highlighted in several papers about elliptic operators with complex coefficients, which can be clearly interpreted as vector-valued operators with real coefficients.
For e.g., it is known that if  $\Omega$ is an open subset of $\Rd$ and  $Q\in L^\infty(\Omega, {\mathbb C}^{d\times d})$ is a matrix-valued function such that the essential infimum of the minimum eigenvalue of the matrix $Q(x)$, when $x$ varies in $\Omega$, is positive, then  there exists $\varepsilon_0=\varepsilon_0(Q)>0$ such that $A= {\rm div}(Q\nabla )$, endowed with suitable boundary conditions, generates an analytic semigroup in $L^p(\Omega, \mathbb C)$ if 
\begin{align*}
\left\vert \frac 1 p -\frac 1 2\right\vert \leq \frac 1 d +\varepsilon_0
\end{align*}
(see \cite{Au,davies_1995, egert1,Tolk}).
In general this condition is sharp in the sense that, if 
$\Big |\frac{1}{p} -\frac{1}{2}\Big |>\frac{1}{d}$, 
then there exists a matrix-valued function $Q$ such that $A$ does not generate a semigroup in $L^p(\Omega)$ (also when $\Omega=\Rd$), see \cite{HMM}.
In order to  give lower bounds for $\varepsilon_0$, the notion of $p$-ellipticity for the matrix-valued function $Q$ was introduced  in \cite{CD}, and the estimates therein contained  were improved and generalized  e.g., in \cite{egert, TerElst}. 
We stress that the operators in the previous papers are pure second-order operators with bounded coefficients.

Various cases of operators with unbounded coefficients and coupled up to the first order   have been studied in the $L^p$-setting, adopting diversified techniques, e.g., a Dore-Venni type theorem on sums of noncommuting operators  due to Monniaux and Pr\"uss  (\cite{monniaux-pruss}) in \cite{HLPRS,KLMR, KMR}, a scalar perturbation argument in the potential in  \cite{AngLorMan, MR}, an extrapolation argument from the setting  of bounded and continuous functions in \cite{AngLorPal}. In \cite{ALMR}, using  quantitative assumptions on the coefficients that allow to suitably control  the growth of the coefficients of the operator in terms of a scalar smooth function $v$,  the generation of an analytic semigroup,  a description of the domain of the generator and integrability conditions on the semigroup are proved. Moreover, in \cite{ALMR1} Gaussian estimates are provided. We refer to  \cite{ALMR} for a detailed and exhaustive comparison of all these results. 

To the authors' best knowledge, the case of operators  with second-order coupling and  unbounded coefficients  has not been yet considered in the literature. 

As is to be expected in view of the preceding considerations, the results in this paper will hold true for $p$ varying in a bounded neighbourhood of $2$, which depends on the coefficients of the operator, and, unless considering operators coupled up to  the first order, it will be not possible to extend the results to the whole interval $(1,\infty)$.

As in the scalar case, also the $L^p$-spaces related to the so-called systems of invariant measures are of particular interest. Unfortunately, at the best of our knowledge only few results are available in the literature (see \cite{AAL_Inv,AAL_Inv1,AL}).

The paper  is organized as follows.
The second section is devoted to some preparatory results that however have their own independent interest. Indeed, we prove local  regularity results for the distributional solutions of elliptic systems with possibly unbounded coefficients. In  the third section the main assumptions on operators \eqref{operator}  are introduced: we assume that the matrices appearing in the  second-order part of the operator are  tamed, in the sense of sesquilinear forms, by a single definite positive matrix.  The hypotheses allow to use a suitable Cauchy-Schwarz inequality that will be crucial in proving  that   operator \eqref{operator}, endowed with its maximal domain,  generates an analytic semigroup in $L^p(\Rd, \Cm)$, with $p$ satisfying the condition 
\begin{eqnarray*}
\left\vert \frac 1 p -\frac 1 2\right\vert\leq K,
\end{eqnarray*}
where $K$ is a constant depending on the coefficients of the second-order part of \eqref{operator}.
It is worth observing that our techniques differ from those adopted in \cite{CD, egert, TerElst}, where  the authors proceed with an extrapolation argument from the space $L^2$ using  Sobolev embeddings. 

Assuming further hypotheses on the growth of the coefficients of \eqref{operator}, it is possible to show that actually the maximal domain coincides with the minimal domain. This is the content of Section \ref{sect-4}. Section \ref{sect-5} is devoted to describe some classes of examples. Finally, two appendices with auxiliary results from Linear Algebra and some technicalities close the paper.

\medskip

\noindent
{\bf Notation.}
Let $d, m\in\N$ and let $\mathbb K=\mathbb R$ or $\mathbb K=\mathbb C$. We denote  by $(\cdot, \cdot)$ and by $|\cdot|$, respectively, the Euclidean inner product and the Euclidean norm in $\mathbb K^m$.
Vector-valued functions are displayed in bold style. Given a function $\uu: \Omega \subseteq \Rd  \rightarrow {\mathbb K}^m$, we denote by $u_k$ its $k$-th component. For every $p\in [1,\infty)$,  $L^p(\Rd, \mathbb K^m)$  denotes the classical vector-valued Lebesgue space endowed with the norm
$\|\f\|_p=(\int_{\Rd} |\f(x)|^pdx)^{1/p}$.
The canonical pairing between $L^p(\Rd, \mathbb K^m)$ and $L^{p'}(\Rd, \mathbb K^m)$ ($p'$ being the index conjugate to $p$), i.e., the integral over $\Rd$ of the function
$x\mapsto (\uu(x), \vv(x))$ when $\uu\in L^p(\Rd, \mathbb K^m)$ and $\vv\in L^{p'}(\Rd, \mathbb K^m)$, is denoted by $\langle \uu,\vv\rangle_{p,p'}$.
For $k\in\mathbb N$, $W^{k,p}(\Rd, \mathbb K^m)$ is the classical vector-valued Sobolev space, i.e., the space of all functions $\uu\in L^p(\Rd, \mathbb K^m)$ whose components have distributional derivatives up to the order $k$, which belong to $L^p(\R^d, \mathbb K^m)$. The norm of $W^{k,p}(\Rd, \mathbb K^m)$ is
denoted by $\|\cdot\|_{k,p}$. When $\mathbb K=\R$ and $m=1$, we simply write $L^p(\Rd)$ and $W^{k,p}(\Rd)$.
By $C^{\infty}_c(\Rd;\mathbb K^m)$, we denote the set all the vector-valued functions which have compact support in $\Rd$ and are infinitely many times differentiable. Similarly, for every $k\in\N$,
$C^k_c(\Rd;\mathbb K^m)$ denotes the set of all the compactly supported functions $\uu:\Rd\to\mathbb K^m$ which are continuously differentiable on $\Rd$ up to the $k$-th order.
We use the subscript ``$b$'' to stress that the functions that we consider are bounded on $\Rd$, together with their derivatives up to the order $k$. When $\mathbb K=\R$ and $m=1$, we simply write $C^{\infty}_c(\Rd)$ and $C^k_c(\Rd)$. If $X(\Rd;\mathbb{K}^m)$ is one of the functional spaces above, we use the notation $X_{\rm{loc}}(\Rd;\mathbb{K}^m)$ to denote the set of functions which belong to $X(\mathcal{K};\mathbb{K}^m)$ for every compact set $\mathcal{K}\subset \Rd$.
           
Finally, given a vector-valued function $\uu$ and $\varepsilon>0$, we denote by $|\uu|_{\varepsilon}$ the real-valued function defined by
\begin{eqnarray*}
|\uu|_{\varepsilon}=\left\{\begin{array}{ll}
(|\uu|^2+\varepsilon)^{\frac{1}{2}},\qquad\;\, &p \in (1,2),\\[1mm]
|\uu|,\qquad\;\, &p \in [2,+\infty).
\end{array}
\right.
\end{eqnarray*}

\section{Preliminary results}
\label{sect-2}
Before stating our main assumptions on the coefficients of the operator $\A$ in \eqref{operator}, we provide some regularity results for distributional solutions to systems of elliptic equations. The scalar counterpart of such results can be found, for instance, in \cite[Theorem D.1.4]{newbook}.

\begin{prop}\label{reg_thm}
Fix $p \in (1,\infty)$ and assume that the diffusion coefficients of the operator \eqref{operator} satisfy the Legendre-Hadamard condition. Then, the following properties are satisfied:
\begin{enumerate}[\rm(i)]
\item 
if $q^{hk}_{ij}\in C^1_b(\Rd)$,
$v_{ij}\in L^{\infty}(\Rd)$ for every $h,k=1,\ldots.d$, $i,j=1,\ldots,m$, and $\uu\in L^p(\Rd;\R^m)$ satisfies the estimate
\begin{equation}\label{new}
\bigg |\int_{\Rd}(\uu,\A{\bm\varphi}) dx\bigg |\le C\|{\bm\varphi}\|_{W^{1,p'}(\Rd;\Rm)}
\end{equation}
for every ${\bm\varphi} \in C^\infty_c(\Rd;\Rm)$ and some positive constant $C$, independent of $\bm \varphi$, then $\uu$ belongs to $W^{1,p}(\Rd;\Rm)$;
\item 
if $q^{hk}_{ij}\in C^1(\Rd)$, $v_{ij}\in L^{\infty}_{\rm loc}(\Rd)$, for every $h,k=1,\ldots,d$, $i,j=1,\ldots,m$, and $\uu\in L^p_{\rm loc}(\Rd;\R^m)$ satisfies estimate \eqref{new} for every $\bm{\varphi}\in C^{\infty}_c(\Rd;\Rm)$, then $\uu$ belongs to $W^{1,p}_{\rm loc}(\Rd;\Rm)$;
\item 
if $q^{hk}_{ij}\in C^1_b(\Rd)$, $v_{ij}\in L^{\infty}(\Rd)$, for every $h,k=1,\ldots,d$ $i,j=1,\ldots,m$, and $\f,\uu\in L^p(\Rd;\R^m)$ satisfy the condition
\begin{equation}\label{weak}
\int_{\Rd}(\uu, \A \bm \varphi) dx=\int_{\Rd}(\f,\bm \varphi)dx
\end{equation}
for every $\bm \varphi \in C^\infty_c(\Rd;\Rm)$, then $\uu\in W^{2,p}(\Rd;\Rm)$;
\item 
if $q^{hk}_{ij}\in C^1(\Rd)$,
 $v_{ij}\in L^{\infty}_{\rm loc}(\Rd)$ for every $h,k=1, \ldots,d$ and $i,j=1,\ldots,m$,  and $\f,\uu\in L^p_{\rm loc}(\Rd;\R^m)$ satisfy equation \eqref{weak} for any $\bm \varphi \in C^\infty_c(\Rd;\Rm)$, then $\uu\in W^{2,p}_{\rm loc}(\Rd;\Rm)$.
\end{enumerate}
\end{prop}

\begin{proof}
(i) Let $\tilde\A_0:W^{2,p'}(\Rd;\Rm)\to L^{p'}(\Rd;\Rm)$ be the operator defined by
\begin{equation*}
\tilde\A_0\uu=\sum_{h,k=1}^d Q^{hk}D_{hk}\uu
\end{equation*}
for every $\uu\in W^{2,p'}(\Rd;\Rm)$.
From estimate \eqref{new} and the boundedness of the coefficients $v_{ij}$, we infer that
\begin{align}
\bigg |\int_{\Rd}(\uu, \tilde\A_0 {\bm \varphi})dx\bigg |\le& 
\bigg |\int_{\Rd}(\uu, \A {\bm \varphi})dx\bigg |
+\bigg |\int_{\Rd}\sum_{h,k=1}^d(\uu, D_hQ^{hk}D_k\bm\varphi)dx\bigg |\notag\\
&+ \bigg |\int_{\Rd} (\uu,V{\bm \varphi}) dx\bigg |\notag\\
\le & C_1\|{\bm \varphi}\|_{W^{1,p'}(\Rd;\Rm)}
\label{uno}
\end{align}
for every ${\bm \varphi}\in C^\infty_c(\Rd;\Rm)$ and a positive constant $C_1$, independent of $\bm \varphi$.

For every $s\in \R^d$ we set $\tau_s \uu=|s|^{-1}(\uu(\cdot+s)-\uu)$. A straightforward change of variables shows that
\begin{align*}
\int_{\Rd}( \tau_s \uu, \widetilde\A_0 {\bm\varphi}) dx= |s|^{-1}\sum_{i,j=1}^m\sum_{h,k=1}^d\int_{\Rd}u_i[q_{ij}^{hk}(\cdot-s)D_{hk}\varphi_j(\cdot-s)-q_{ij}^{hk}D_{hk}\varphi_j] dx.
\end{align*}
Adding and subtracting the term 
$\displaystyle|s|^{-1}\sum_{i,j=1}^m\sum_{h,k=1}^d\int_{\Rd}u_iq_{ij}^{hk}D_{hk}\varphi_j(\cdot-s)dx$
and using \eqref{uno} and the fact that the coefficients $q^{hk}_{ij}$ belong to $C^1_b(\Rd)$,
we obtain
\begin{align*}
\bigg |\int_{\Rd}( \tau_s \uu,\widetilde\A_0 {\bm\varphi}) dx\bigg |\le & \bigg |\int_{\Rd}(\uu, \widetilde\A_0 ({\tau_{-s}\bm\varphi)})dx\bigg |\\
&+\bigg |\int_{\Rd}\bigg( \uu, \sum_{h,k=1}^d( \tau_{-s}Q^{hk})D_{hk}{\bm \varphi}(\cdot-s)\bigg) dx\bigg|\\
\le& C_1\|\tau_{-s}\bm \varphi\|_{W^{1,p'}(\Rd;\R^m)}\\
&+\|\uu\|_{L^p(\Rd;\R^m)}\sum_{h,k=1}^d\|\tau_{-s}Q^{hk}\|_\infty\|D_{hk}\bm \varphi\|_{L^{p'}(\Rd;\R^m)}\\
\le& C_2\|{\bm\varphi}\|_{W^{2,p'}(\Rd;\Rm)}
\end{align*}
for every ${\bm \varphi}\in C^\infty_c(\Rd;\Rm)$ and some positive constant $C_2$ independent of $s$ and $\bm{\varphi}$. Thus,
\begin{align}\label{befana}
\bigg |\int_{\Rd}(\tau_s \uu,\lambda{\bm \varphi}-\widetilde\A_0{\bm \varphi}) dx\bigg |=&\bigg |\int_{\Rd}[\lambda( \uu,\tau_{-s} {\bm \varphi})-(\tau_s \uu,\widetilde\A_0{\bm \varphi})] dx\bigg |\notag\\
\le & C_3\|{\bm\varphi}\|_{W^{2,p'}(\Rd;\Rm)}
\end{align}
for every ${\bm \varphi}\in C^\infty_c(\Rd;\Rm)$, every $\lambda>0$ and some positive constant $C_3$, independent of $\bm\varphi$ and $s$.

Estimate \eqref{befana} can be extended by density to any $\bm\varphi\in W^{2,p'}(\Rd;\R^m)$. Thus, since the operator $\lambda-\widetilde\A_0$ is invertible from $W^{2,p'}(\Rd;\R^m)$ to $L^{p'}(\Rd;\Rm)$ for a large $\lambda$, thanks to \cite[Theorem 2.1]{miyazaki}, we can take $\bm \varphi=(\lambda-\widetilde\A_0)^{-1}(\tau_s\uu|\tau_s \uu|^{p-2})$ and get
\begin{equation}\label{bef1}
\|\bm \varphi\|_{W^{2,p'}(\Rd;\R^m)}\le C_3\|\tau_s\uu|\tau_s \uu|^{p-2}\|_{L^{p'}(\Rd;\Rm)}= C_4\|\tau_s \uu\|_{L^p(\Rd;\R^m)}^{p-1}
\end{equation}
for some positive constants $C_3$ and  $C_4$, independent of $s$. Replacing such a function $\bm \varphi$ in \eqref{befana}, thanks to estimate \eqref{bef1} we deduce that
$\|\tau_s \uu\|_{L^p(\Rd;\R^m)}\le C_5$
for some positive constant $C_5$, independent of $s$. Consequently $\uu\in W^{1,p}(\Rd;\Rm)$.

(ii) Fix ${\bm \varphi}\in C^\infty_c(\Rd;\Rm)$ and, for a fixed $r>0$, let us consider a function $\psi_r\in C^\infty_c(\Rd)$ such that $\chi_{B(0,r)}\le \psi_r \le \chi_{B(0,2r)}$. It is straightforward to deduce that
\begin{equation}\label{est}\psi_r \A{\bm \varphi}=\A(\psi_r \bm \varphi)-\sum_{h,k=1}^d Q^{hk}D_k {\bm \varphi}D_h \psi_r
\end{equation}
so that
\begin{align*}
\int_{\Rd}( \uu,\psi_r \A{\bm \varphi}) dx= \int_{\Rd} \bigg( \uu,\A(\psi_r \bm \varphi)-\sum_{h,k=1}^d Q^{hk}D_k {\bm \varphi}D_h \psi_r\bigg) dx
\end{align*}
and taking \eqref{new} and the local boundedness of the coefficients of the operator $\A$ into account, we get
\begin{align*}
\bigg|\int_{\Rd}(\uu,\psi_r \A{\bm \varphi})dx\bigg |
& \le \bigg |\int_{\Rd}(\uu,\A(\psi_r \bm \varphi)) dx\bigg |+\bigg |\int_{\Rd} \bigg (\uu,\sum_{h,k=1}^d Q^{hk}D_k {\bm \varphi}D_h \psi_r\bigg ) dx\bigg |\\
& \le C_6\|\bm \varphi\|_{W^{1,p'}(\Rd;\Rm)}
\end{align*}
for some positive constant $C_6$, independent of $\bm\varphi$.

Now, let us consider a function $\eta\in C^\infty_c(\Rd)$ such that $\chi_{B(0,2r)}\le \eta \le \chi_{B(0,4r)}$ and set
\begin{equation}\label{tilde}\tilde {Q}^{hk}= \eta Q^{hk}+(1-\eta)\delta_{hk}I_m,\quad\;\,\quad \tilde{V}=\eta V
\end{equation}
for any $h,k=1, \ldots,d$, where $I_m$ denotes the $m\times m$ identity matrix. Clearly, the coefficients of the operator
\begin{eqnarray*}
\tilde{\A}=\sum_{h,k=1}^d D_h(\tilde{Q}^{hk}D_{k})
-\tilde{V}
\end{eqnarray*}
satisfy the assumptions in (i) and, since
\begin{eqnarray*}
\int_{\Rd}( \uu,\psi_r \A{\bm \varphi}) dx= \int_{\Rd}(\uu,\psi_r \tilde{\A}{\bm \varphi}) dx=\int_{\Rd}(\psi_r\uu, \tilde{\A}{\bm \varphi}) dx,
\end{eqnarray*}
applying property (i), we deduce that $\uu\psi_r \in W^{1,p}(\Rd;\Rm)$, that implies that $\uu \in W^{1,p}(B(0,r);\Rm)$. By the arbitrariness of $r>0$ we conclude that $\uu \in W^{1,p}_{\rm loc}(\Rd;\R^m)$.

(iii) Starting from \eqref{weak}, we obtain
\begin{align}\label{car}
\int_{\Rd}( \uu, \A_0 \bm \varphi) dx
= \int_{\Rd}( \f_1,\bm \varphi) dx
\end{align}
for any $\bm\varphi \in C^\infty_c(\Rd;\Rm)$,
where $\f_1=\f+V^T\uu$ belongs to $L^p(\Rd;\Rm)$.
Moreover, from the previous identity we get
\begin{equation}\label{cruc}
\int_{\Rd}(\uu, \lambda \bm \varphi-\A_0 \bm \varphi) dx= \int_{\Rd}(\lambda \uu-\f_1, \bm \varphi) dx=: \int_{\Rd}(\g, \bm \varphi) dx
\end{equation}
for any $\bm \varphi \in C^\infty_c(\Rd;\R^m)$ and, by density, for any $\bm \varphi \in W^{2,p'}(\Rd;\R^m)$.
Now, fix $\lambda>0$ in the resolvent sets of both $\A_0$ and of its adjoint $\A_0^*$, defined by
\begin{equation*}
\A_0^*\uu=\sum_{h,k=1}^d D_h((Q^{kh})^TD_{k}\uu)
\label{lead*}
\end{equation*}
(to which the results in \cite[Theorem 2.1]{miyazaki} can be applied). Then, to prove the claim, after observing that $\g \in L^p(\Rd;\Rm)$ we show that $\uu= (\lambda-\A_0^*)^{-1}\g$. In this case $\uu$ will belong to the domain of the realization of $\A_0^*$ in $L^p(\Rd;\Rm)$, which is $W^{2,p}(\Rd;\R^m)$, thanks again to \cite[Theorem 2.1]{miyazaki}. To this aim we set $\z= \uu-(\lambda-\A_0^*)^{-1}\g$ and observe that by \eqref{cruc} it holds that
\begin{equation*}
\int_{\Rd}( \z, \lambda  \bm \varphi-\A_0 \bm \varphi ) dx=0
\end{equation*}
for any  $\bm \varphi \in W^{2,p'}(\Rd;\R^m)$. The surjectivity of $\lambda I-\A_0$ as a map from $W^{2,p'}(\Rd;\Rm)$ into $L^{p'}(\Rd;\Rm)$ allows us to conclude that $\z \equiv \bm 0$.

(iv) Formula \eqref{weak} together with the assertion in (ii) yield immediately that $\uu \in W^{1,p}_{\rm loc}(\Rd;\Rm)$. Further, formula \eqref{car} continues to hold for any $\bm \varphi \in C^\infty_c(\Rd;\R^m)$ with $\f_1$ that now belongs to $L^p_{\rm loc}(\Rd;\Rm)$. We set $\vv=\psi_r\uu$, where $\psi_r$ is defined in the proof of claim (ii) and observe that formula \eqref{est} holds true also with $\A_0$ in place of $\A$.
Thus, using \eqref{car} and the integration by parts formula, we get
\begin{align*}
&\int_{\Rd}( \lambda \bm \varphi-\A_0 \bm \varphi, \vv) dx\\
=& \int_{\Rd} \lambda( \bm \varphi, \vv) - ( \psi_r\A_0 \bm \varphi, \uu) dx\\
=& \int_{\Rd} \lambda( \bm \varphi, \vv) - \Big( \A_0(\psi_r \bm \varphi)-\sum_{h,k=1}^d Q^{hk}D_k \bm \varphi D_h \psi_r, \uu\Big) dx\\
=& \int_{\Rd} (\bm \varphi, \lambda\vv - \psi_r \f_1 ) dx+\int_{\Rd}\bigg(\sum_{h,k=1}^d  Q^{hk}D_k {\bm \varphi}D_h \psi_r, \uu\bigg) dx\\
=& \int_{\Rd} ( \bm \varphi, \lambda\vv - \psi_r \f_1 ) dx-\int_{\Rd}\bigg( {\bm \varphi}, \sum_{h,k=1}^d D_k( (Q^{hk})^T\uu D_h \psi_r)\bigg) dx\\
=&\!: \int_{\Rd}(\bm \varphi, \lambda\vv - \psi_r \f_1 +\f_2) dx,
\end{align*}
whence, since all the functions under the integral sign are supported on $B(0,2r)$,
\begin{eqnarray*}
\int_{\Rd}(\vv, \A_1 \bm \varphi) dx=\int_{\Rd} (\bm \varphi, \psi_r \f_1 -\f_2) dx,
\end{eqnarray*}
where ${\A}_1= \sum_{h,k=1}^d D_h(\tilde{Q}^{hk}D_{k})$ and $\tilde{Q}^{hk}$ are the matrix-valued functions defined in \eqref{tilde}.
Note that the function $ \psi_r \f_1 -\f_2$ belongs to $L^p(\Rd;\R^m)$, thus, using property (iii), we conclude that $\vv$ belongs to $W^{2,p}(\Rd;\R^m)$ or, equivalently, that $\uu \in W^{2,p}(B(0,r);\R^m)$. The arbitrariness of $r>0$ shows that $\uu\in W^{2,p}_{\rm loc}(\Rd;\R^m)$, and we are done.
\end{proof}

\section{Assumptions and main results}
\label{sect-3}
In this section we state the main assumptions on the coefficients of the operators $\A$ defined in \eqref{operator}. 
\begin{hyp}\label{hyp_0}
\renewcommand{\labelitemi}{\normalfont -}
\begin{enumerate}[\rm(i)]
\item 
For every $h,k=1,\ldots,d$, the matrix-valued function $Q^{hk}=q_{hk}I+A^{hk}$ satisfies the following conditions:
\begin{itemize}
\item[$\diamond$]
$q_{hk}\in C^1(\Rd)$ and the matrix-valued function $Q=(q_{hk})$ satisfies the condition  
${\rm Re}\,(Q(x)\xi, \xi) >0$ for every $x\in\Rd$ and $\xi\in \mathbb C^d\setminus\{0\}$. Moreover, there exists a positive constant $c_0$ such that
\begin{equation}\label{Im_Q}
|({\rm Im}(Q(x)\xi,\xi)|\le c_0{\rm Re}(Q(x)\xi,\xi),\qquad\;\, x\in \Rd, \xi\in \mathbb C^d;
\end{equation}
\item[$\diamond$]
the matrix-valued functions $A^{hk}=(a^{hk}_{ij})$ have entries $a^{hk}_{ij}\in C^1(\Rd)$, for every $i,j=1,\ldots,m$ and $h,k=1,\ldots,d$, and satisfy the following condition:
\begin{align}\label{realLegendrefunc}
0\leq  {\rm Re}\sum_{h,k=1}^d (A^{hk}(x)\theta^k, \theta^h) \leq {\mathscr C}\,{\rm Re} \sum_{i=1}^m \sum_{h,k=1}^dq_{hk}(x)\theta^k_i\overline{\theta^h_i}
\end{align}
 for some constant ${\mathscr C}>0$, every $\theta^1,\ldots,\theta^d\in\Cm$ and $x\in\Rd$;
 \end{itemize}
\item
for every $\theta^1,\ldots,\theta^d\in \Cm$ and $x\in\Rd$ it holds that
\begin{align}
\bigg | {\rm Im}\sum_{h,k=1}^d (A^{hk}(x)\theta^k, \theta^h)\bigg |\leq {\mathscr C}\, {\rm Re}\sum_{i=1}^m \sum_{h,k=1}^dq_{hk}(x)\theta^k_i\overline{\theta^h_i},
\label{comLeg}
\end{align}
where ${\mathscr C}$ is the constant in \eqref{realLegendrefunc};
\item
$v_{ij}\in L^\infty_{\rm{loc}}(\Rd)$ for every $i,j=1, \ldots, m$ and
${\rm Re}(V(x) \xi, \xi)\geq 0$ for every $x\in \Rd$ and $\xi \in \Cm$.
\end{enumerate}
\end{hyp}

\begin{rmk}{\rm 
 Note that Hypothesis \ref{hyp_0}(i) implies the Legendre-Hadamard condition, i.e.,
 \begin{equation*}
 \sum_{h,k=1}^d\sum_{i,j=1}^m Q^{hk}_{ij}\xi_i\xi_j\eta_h\eta_k\ge 0, \qquad\;\, \xi \in \R^m,\;\, \eta \in \R^d,
\end{equation*}
which is usually assumed in the classical theory of strongly coupled systems of elliptic equations with bounded coefficients. Moreover, assuming further Hypothesis \ref{hyp_0}(ii), Proposition \ref{CSconf} yields the Cauchy-Schwarz inequalities
\begin{equation}\label{C-S_0}
|(Q(x) \xi,\zeta)|\le (1+c_0)\left ({\rm Re} (Q(x)\xi,\xi\right))^{\frac{1}{2}}\left ({\rm Re}(Q(x)\zeta,\zeta)\right )^{\frac{1}{2}}
\end{equation}
and
\begin{equation}\label{C-S}
\bigg |\sum_{h,k=1}^d(A^{hk}(x)\vartheta^k,\eta^h)\bigg |\le 2{\mathscr C} \bigg ({\rm Re}\sum_{i=1}^m\sum_{h,k=1}^d q_{hk}(x)\vartheta^k_i\overline{\vartheta^h_i}\bigg )^{\frac{1}{2}}\bigg ({\rm Re}\sum_{i=1}^m\sum_{h,k=1}^d q_{hk}(x)\eta^k_i\overline{\eta^h_i}\bigg )^{\frac{1}{2}}
\end{equation}
for every $\xi,\zeta \in \mathbb{C}^d$, $\vartheta^1,\ldots\vartheta^d,\eta^1,\ldots,\eta^d\in \Cm$ and $x\in\Rd$. }
\end{rmk}

In what follows, we will use the following formulas, which hold true for every smooth function $\uu\in C_c^\infty (\Rd, \Cm)$:
\begin{align}\label{Dmod}
&D_h |\uu|^2 = 2{\rm Re}\sum_{i=1}^m u_iD_h\overline{ u_i}=2{\rm Re}(\uu, D_h\uu),\\
&\label{disDmod}(Q\nabla|\uu|^2, \nabla|\uu|^2)={\rm Re}(Q\nabla|\uu|^2, \nabla|\uu|^2)\leq 4 |\uu|^2 \sum_{i=1}^m  {\rm Re}(Q\nabla u_i, \nabla u_i).
\end{align}

\begin{lemm}\label{lp_diss}
Let us assume that Hypotheses $\ref{hyp_0}$ are satisfied. If $p \in J$, where
\begin{equation*}
J:=\left\{\begin{array}{ll}
\displaystyle{\left[2-\frac{1}{2{\mathscr C}+1}, 2+\frac{1}{{\mathscr C}^2}\right]}, \qquad\;\, & {\mathscr C}\in (0,1),\\
[4mm]
\displaystyle{\left[2-\frac{1}{2{\mathscr C}+1}, 2+\frac{1}{2{\mathscr C}-1}\right]}, \qquad\;\, &{\mathscr C} \in [1,+\infty),
\end{array}
\right.
\end{equation*}
then $(\A, C_c^{\infty}(\Rd;\Rm))$ is $L^p$-dissipative, i.e., for every $\uu\in C^\infty_c(\Rd;\mathbb{C}^m)$ 
it holds that
\begin{equation}
{\rm Re}\int_{\Rd}(\A\uu,\uu)|\uu|^{p-2}dx\le 0.
\label{luca_0}
\end{equation}

Moreover, if $p$ belongs to 
\begin{equation}
\widetilde J:=\left\{\begin{array}{ll}
\displaystyle{\left[2-\frac{1}{2{\mathscr C}+1}, 2+\frac{1}{{\mathscr C}^2}\right )}, \qquad\;\, & {\mathscr C}\in (0,1),\\
[4mm]
\displaystyle{\left[2-\frac{1}{2{\mathscr C}+1}, 2+\frac{1}{2{\mathscr C}-1}\right )}, \qquad\;\, &{\mathscr C} \in [1,+\infty),
\end{array}
\right.
\label{luca_0000}
\end{equation}
then there exists a positive constant $\delta$ such that
\begin{equation}
{\rm Re}\int_{\Rd}(\A\uu,\uu)|\uu|_\varepsilon^{p-2}dx\le -\delta\sum_{i=1}^m\int_{\Rd}{\rm Re}(Q\nabla u_i,\nabla u_i) |\uu|_{\varepsilon}^{p-2}dx
\label{luca_00}
\end{equation}
for every $\uu\in C^{\infty}_c(\Rd;\Cm)$ and $\varepsilon>0$.
\end{lemm}

\begin{proof}
Fix $\uu\in C^\infty_c(\Rd;\mathbb{C}^m)$. 
A straightforward computation shows that
\begin{align}
&\int_{\Rd}(\A\uu,\uu|\uu|_\varepsilon^{p-2}) dx\notag\\
=& -\int_{\Rd}\sum_{h,k=1}^d (Q^{hk}D_k \uu, D_h\uu) |\uu|_\varepsilon^{p-2}dx\notag\\
&- \frac{p-2}{2} \int_{\Rd}\sum_{h,k=1}^d (Q^{hk}D_k \uu, \uu)|\uu|_\varepsilon^{p-4} D_h|\uu|^2 dx-\int_{\Rd}(V\uu, \uu) |\uu|_\varepsilon^{p-2}dx\notag\\
=& - \int_{\Rd}\sum_{i=1}^m(Q \nabla u_i, \nabla u_i)|\uu|_\varepsilon^{p-2}dx-\int_{\Rd}\sum_{h,k=1}^d (A^{hk}D_k \uu, D_h\uu) |\uu|_\varepsilon^{p-2}dx\notag\\
&- \frac{p-2}{2}\int_{\Rd}\sum_{i=1}^m(Q\nabla u_i,\nabla|\uu|^2)\overline{u_i}|\uu|_\varepsilon^{p-4}dx
-\int_{\Rd}(V\uu,\uu)
|\uu|_\varepsilon^{p-2}dx\notag\\
&- \frac{p-2}{2} \int_{\Rd}\sum_{h,k=1}^d (A^{hk}D_k \uu, \uu)|\uu|_\varepsilon^{p-4}D_h|\uu|^2 dx.
\label{formula}
\end{align}
Hence, using formula \eqref{Dmod} and Hypothesis \ref{hyp_0}(i) and (iii), we can estimate
\begin{align}
&{\rm Re}\int_{\Rd}(\A\uu,\uu|\uu|_\varepsilon^{p-2})dx\notag\\
\le & - \int_{\Rd}{\rm Re} \sum_{i=1}^m (Q\nabla u_i, \nabla u_i)|\uu|_{\varepsilon}^{p-2} dx - \frac{p-2}{4}\int_{\Rd} (Q\nabla|\uu|^2,\nabla  |\uu|^2)  |\uu|_\varepsilon^{p-4}dx\notag\\
 & - \frac{p-2}{2} \int_{\Rd}{\rm Re}\sum_{h,k=1}^d (A^{hk}D_k \uu, \uu )|\uu|_\varepsilon^{p-4} D_h|\uu|^2dx.
 \label{form-1}
\end{align}

By \eqref{C-S} and Young inequality, we get 
\begin{align}
&\bigg|\int_{\Rd}{\rm Re}\sum_{h,k=1}^d (A^{hk}D_k \uu, \uu) |\uu|_\varepsilon^{p-4}D_h|\uu|^2dx\bigg |\notag\\
=& \bigg |\int_{\Rd}{\rm Re}\sum_{h,k=1}^d(A^{hk}(D_k \uu) |\uu|_\varepsilon^{\frac{p-2}{2}} , (|\uu|_\varepsilon^{\frac{p-6}{2}}D_h|\uu|^2)\uu )dx  \bigg |\notag\\
\le & 2{\mathscr C}\bigg( \int_{\Rd}{\rm Re}\sum_{i=1}^m(Q\nabla u_i,\nabla u_i)|\uu|_\varepsilon^{p-2} dx\bigg)^{\frac 1 2}
\bigg( \int_{\Rd}(Q\nabla |\uu|^2,\nabla|\uu|^2)|\uu|^2 |\uu|_\varepsilon^{p-6} dx\bigg)^{\frac 1 2}\notag\\
\le & 2{\mathscr C}\sigma  \int_{\Rd}{\rm Re} \sum_{i=1}^m (Q\nabla u_i, \nabla u_i)|\uu|_{\varepsilon}^{p-2} dx + \frac{{\mathscr C}}{2\sigma} \int_{\Rd} (Q\nabla|\uu|^2,\nabla  |\uu|^2)  |\uu|_\varepsilon^{p-4}dx
\label{form-2}
\end{align}
for every $\sigma>0$. From \eqref{form-1} and \eqref{form-2}, it follows that
\begin{align}
&{\rm Re}\int_{\Rd}(\A\uu,\uu)|\uu|_\varepsilon^{p-2}dx\notag\\ 
\leq &(- 1+|p-2|{\mathscr C}\sigma)\int_{\Rd}\sum_{i=1}^m{\rm Re}(Q\nabla u_i, \nabla u_i)|\uu|_\varepsilon^{p-2} dx\notag\\
&+\bigg (\frac{|p-2|{\mathscr C} }{4\sigma}-\frac{p-2}{4} \bigg ) \int_{\Rd}(Q\nabla|\uu|^2,\nabla|\uu|^2)|\uu|_\varepsilon^{p-4} dx=: g_p(\sigma).
\label{form-3}
\end{align}
Now, let us set  $$A:=\sum_{i=1}^m\int_{\Rd}{\rm Re}(Q\nabla u_i, \nabla u_i)|\uu|_\varepsilon^{p-2} dx, \qquad\,\,B:=\int_{\Rd}(Q\nabla|\uu|^2,\nabla|\uu|^2)|\uu|_\varepsilon^{p-4} dx$$
and observe that $B\le 4A$, thanks to estimate \eqref{disDmod}. Therefore, if $A=0$ then $B=0$ as well.
In this case, the right hand-side in \eqref{form-3} is identically zero for any $\sigma>0$. In the non trivial case, i.e., $A\neq 0$, then
\begin{align*}
\min\{ g_p(\sigma): \sigma>0\}= g_p\left(\sqrt{\frac{B}{4A}}\right)
= -A -\frac{p-2}{4}B+|p-2|{\mathscr C}\sqrt{AB}.
\end{align*}

Now, we look for the values of $p$ such that
$g_{p}\left(\sqrt{\frac{B}{4A}}\right)\leq -\delta A$ for any $A,B>0$, such that $B\leq 4 A$, and some constant $\delta\in [0,1)$. Clearly, if $B=0$, then $\min g_p\leq -A$. Otherwise, if $B>0$, then  
setting $t=\sqrt{\frac A B}$  we need to study the inequality
\begin{equation}
\label{dis}
(1-\delta)t^2 - |p-2| {\mathscr C} t + \frac{p-2}{4}\geq 0, \qquad \;\,t\geq  \frac{1}{2}.
\end{equation}
If ${\mathscr C}^2(p-2)^2 - (1-\delta)(p-2)\leq 0$, i.e., $0\leq p-2 \leq (1-\delta){\mathscr C}^{-2}$, then \eqref{dis} is satisfied for any $t\in\R$.
\noindent
If  ${\mathscr C}^2(p-2)^2 - (1-\delta)(p-2)> 0$, then \eqref{dis} is satisfied for any $t\geq \frac 1 2$ if and only if
\begin{equation*}
|p-2|{\mathscr C}+ \sqrt{{\mathscr C}^2(p-2)^2-(1-\delta)(p-2)}\leq 1
\end{equation*}
or, equivalently,
\begin{equation}\label{b3}
\left\{
\begin{array}{ll}
x({\mathscr C}^2x-(1-\delta))\ge 0,\\[1mm]
1-|x|{\mathscr C}\ge 0,\\[1mm]
2|x|{\mathscr C}-(1-\delta)x \le 1,
\end{array}
\right.
\end{equation}
where we set $x=p-2$.
Since the inequality $2|x|{\mathscr C}-(1-\delta)x \le 1$ is satisfied for $x \ge -(2{\mathscr C}+1-\delta)^{-1}$, if ${\mathscr C} \le \frac{1}{2}-\frac{\delta}{2}$, and for $-(2{\mathscr C}+1-\delta)^{-1}\le x\le (2{\mathscr C}-1+\delta)^{-1}$ otherwise, we conclude that
\eqref{b3} is satisfied if and only if
$x \in  \left[-\frac{1}{2{\mathscr C}+1-\delta}, 0\right]$ if ${\mathscr C} \in (0,1-\delta)$ and $x \in \left[-\frac{1}{2{\mathscr C}+1-\delta},0\right ]\cup\left [\frac{1-\delta}{{\mathscr C}^2},\frac{1}{2{\mathscr C}-1+\delta}\right]$ if ${\mathscr C} \ge 1-\delta$.

Adding also the first case and writing the latter conditions in terms of $p$, we conclude that $\min g_p\le-\delta A$  if and only if $p \in J_{\delta}$, where
\begin{equation*}
J_{\delta}=\left\{\begin{array}{ll}
\displaystyle{\left[2-\frac{1}{2{\mathscr C}+1-\delta}, 2+\frac{1-\delta}{{\mathscr C}^2}\right]}, \qquad\;\, & {\mathscr C}\in (0,1-\delta),\\
[4mm]
\displaystyle{\left[2-\frac{1}{2{\mathscr C}+1-\delta}, 2+\frac{1}{2{\mathscr C}-1+\delta}\right]}, \qquad\;\, &{\mathscr C} \in [1-\delta,+\infty).
\end{array}
\right.
\end{equation*}
From the previous computations, it follows that estimate \eqref{luca_0} is satisfied if $p\in J$. On the other hand, if $p$ belongs to $\widetilde J$, then condition \eqref{luca_00} is satisfied for some $\delta>0$.
Indeed, suppose that $\mathscr C\in (0,1)$. Then,
${\mathscr C}\in (0,1-\delta)$ for every $\delta\in (0,1-{\mathscr C})$. The above results show that \eqref{luca_00}
holds true for every 
$p\in\displaystyle\left [2-\frac{1}{2{\mathscr C}+1-\delta}, 2+\frac{1-\delta}{{\mathscr C}^2}\right]$. If
$p\in\displaystyle\left [2-\frac{1}{2{\mathscr C}+1}, 2+\frac{1}{{\mathscr C}^2}\right)$, then we can determine $\delta\in (0,1-{\mathscr C})$ such that $p\in J_{\delta}$ and, consequently,
\eqref{luca_00} follows with this $\delta$. On the other hand, if ${\mathscr C}\ge 1$, then, ${\mathscr C}>1-\delta$ for every $\delta>0$ and \eqref{luca_00} is satisfied for every $p\in J_{\delta}$. If $p\in \displaystyle\left [2-\frac{1}{2{\mathscr C}+1}, 2+\frac{1}{2{\mathscr C}-1}\right)$, then, we can find $\delta>0$ such that
$p\in J_{\delta}$ and
\eqref{luca_00} follows with this $\delta$. The proof is complete.
\end{proof}

Now we prove that the realization of operator $\A$ in $L^p(\Rd;\R^m)$ with domain $D_{p,{\rm max}}(\A)=\{\uu\in L^p(\Rd;\Rm)\cap W^{2,p}_{\rm loc}(\Rd;\Rm):\A\uu\in L^p(\Rd;\Rm)\}$ generates a strongly continuous semigroup of contractions in $L^p(\Rd;\Rm)$.

\begin{thm}\label{mainth}
Assume that Hypotheses $\ref{hyp_0}$ are satisfied and that there exist a positive function $\psi \in C^1(\Rd)$, blowing up at $\infty$, and $K>0$ such that
\begin{eqnarray}\label{ipo_core}
\frac{(Q\nabla \psi, \nabla\psi)}{(\psi\log \psi)^2} \leq K.
\end{eqnarray}
Then, for any $p \in (1,\infty)$, satisfying the condition
\begin{eqnarray}\label{condition_p}
\left\vert \frac 1 p -\frac 1 2\right|\leq\frac{1}{2(4{\mathscr C}+1)},
\end{eqnarray}
the realization ${\bf A}_p$ of the operator $\A$ in $L^p(\Rd;\Rm)$, with domain $D_{p,\rm{max}}(\A)$, generates a strongly continuous semigroup of contraction in $L^p(\Rd;\Rm)$. Moreover, the space $C_c^\infty(\Rd;\Rm)$ is a core of $({\bf A}_p,D_{p,\rm{max}}(\A))$.
\end{thm}

\begin{proof}
Due to its length we split the proof into two steps.

\emph{Step 1}. Here, we prove that $(\A,C_c^{\infty}(\Rd;\Rm))$ is a closable operator in $L^p(\Rd;\Rm)$ and its closure generates a strongly continuous semigroup.

To this aim, first note that, for any ${\mathscr C}>0$, the set of $p$'s satisfying condition \eqref{condition_p} is a subset of the set $\widetilde J$ introduced in Lemma \ref{lp_diss}. Then, $(\A, C_c^{\infty}(\Rd;\Rm))$ is $L^p$-dissipative and the assertion will follow from the Lumer-Phillips theorem  (see e.g., \cite[Theorem 3.15]{engnagel}) if we prove that $(\lambda I- \A)(C_c^\infty(\Rd;\Rm))$ is dense in $L^p(\Rd;\Rm)$ for some  $\lambda >0$. Thus, we fix $\lambda>0$ and $\uu \in L^{p'}(\Rd;\Rm)$ such that
\begin{equation}\label{eq1}
\langle\A\bm\varphi, \uu\rangle_{p,p'}= \lambda \langle \bm{\varphi}, \uu\rangle _{p,p'}
\end{equation}
for every $\bm\varphi \in C^\infty_c(\Rd;\Rm)$. We have to show that $\uu\equiv {\bf 0}$. The main step consists in showing that
\begin{align}
\lambda \int_{\Rd}\zeta_n^2|\uu|^2|\uu|_\varepsilon^{p'-2} dx  \leq
C_*\int_{\Rd}(Q\nabla \zeta_n, \nabla \zeta_n)  |\uu|_\varepsilon^{p'}dx
\label{form-5}
\end{align}
for some positive constant $C_*$ and every $n\in\N$,
where $\zeta_n= \zeta(n^{-1}\log \psi)$ and $\zeta:[0,\infty)\to [0,1]$ is a smooth function such that $\zeta(s)= 1$ if $ s \in [0,1]$ and $\zeta(s)=0$ if $s\in [2,\infty)$.
Once this inequality is proved, the assumption \eqref{ipo_core} will be crucial to conclude. Indeed, letting $\varepsilon$ tend to $0$  and using the dominated convergence theorem we obtain 
\begin{align}\label{def_n}
\lambda \int_{\Rd}\zeta_n^2|\uu|^{p'} dx  &\leq
C_*\int_{\Rd}(Q\nabla \zeta_n, \nabla \zeta_n)  |\uu|^{p'}dx\notag\\
&=\frac{C_*}{n^2}\int_{\Rd}(\zeta'(n^{-1}\log \psi))^2\psi^{-2}( Q \nabla \psi, \nabla \psi) |\uu|^{p'}dx\notag\\
&\leq \frac{C_*K}{n^2}\int_{\Rd}(\zeta'(n^{-1}\log \psi))^2(\log\psi)^2|\uu|^{p'} dx,
\end{align}
thanks to \eqref{ipo_core}.
Thus, since the support of $\zeta'$ is contained in the set $\{x \in \Rd: n\le \log\psi(x)\le 2n\}$, we can use again the dominated convergence theorem to let $n$ tend to $+\infty$ in \eqref{def_n} and deduce that $\lambda \|\uu\|_{L^{p'}(\Rd; \mathbb C^m)}\le 0$, whence $\uu\equiv\bm{0}$.

So, let us prove \eqref{form-5} starting from \eqref{eq1}. Using Proposition \ref{reg_thm}(iv), we deduce that $\uu\in W^{2,p'}_{\rm loc}(\Rd;\R^m)$. 
Clearly, we can extend the validity of \eqref{eq1} to every function $\bm\varphi\in W^{2,p}(\R^d;\R^m)$ with compact support. Thus, we can write \eqref{eq1} with the function $\bm\varphi$ being replaced by $\bm{\varphi}_n:= \zeta_n^2\uu|\uu|_\varepsilon^{p'-2}$. Integrating by parts the second-order term in the left-hand side of such a formula and using the first part of \eqref{realLegendrefunc} and Hypothesis \ref{hyp_0}(iii),  we get
\begin{align}
&\lambda \int_{\Rd}\zeta_n^2|\uu|^2|\uu|_\varepsilon^{p'-2} dx\notag\\
=&-\int_{\Rd}\sum_{i=1}^m(Q\nabla u_i,\nabla u_i) \zeta_n^2|\uu|_\varepsilon^{p'-2}dx
-\frac{p'-2}{2}\int_{\Rd}\sum_{i=1}^m(Q\nabla |\uu|^2,\nabla u_i)u_i\zeta_n^2|\uu|_\varepsilon^{p'-4}dx\notag\\
&-\int_{\Rd}\sum_{i=1}^m(Q\nabla (\zeta_n^2),\nabla u_i)u_i |\uu|_\varepsilon^{p'-2}dx
-\int_{\Rd}\sum_{h,k=1}^d(A^{hk}D_k\uu , D_h\uu)\zeta_n^2|\uu|_\varepsilon^{p'-2}dx\notag\\
&- \frac{p'-2}{2} \int_{\Rd}\sum_{h,k=1}^d (A^{hk}\uu D_k|\uu|^2 , D_h\uu) \zeta_n^2|\uu|_\varepsilon^{p'-4}dx\notag\\
&-\int_{\Rd}\sum_{h,k=1}^d(A^{hk}\uu D_k\zeta_n^2 , D_h\uu) |\uu|_\varepsilon^{p'-2}dx
-\int_{\Rd}\zeta_n^2 ( V \uu, \uu) |\uu|_\varepsilon^{p'-2}dx\notag\\
\leq  &-\int_{\Rd}\sum_{i=1}^m (Q\nabla u_i, \nabla u_i) \zeta_n^2|\uu|_\varepsilon^{p'-2}dx-\frac{p'-2}{4}\int_{\Rd}(Q\nabla |\uu|^2, \nabla |\uu|^2)  \zeta_n^2|\uu|_\varepsilon^{p'-4}dx\notag\\
&-\frac{1}{2} \int_{\Rd}(Q\nabla\zeta_n^2, \nabla |\uu|^2)  |\uu|_\varepsilon^{p'-2}dx-  \int_{\Rd}\sum_{h,k=1}^d (A^{hk}\uu D_k\zeta_n^2 , D_h\uu) |\uu|_\varepsilon^{p'-2}dx\notag\\
&- \frac{p'-2}{2} \int_{\Rd}\sum_{h,k=1}^d (A^{hk}\uu D_k|\uu|^2 , D_h\uu) \zeta_n^2|\uu|_\varepsilon^{p'-4}dx,
\label{form-10}
\end{align}
where we used also the fact that $\nabla|\uu|^2=2\displaystyle\sum_{i=1}^mu_i\nabla u_i$ and $\displaystyle\sum_{h,k=1}^d(A^{hk}D_k\uu,D_h\uu)\ge 0$ on $\Rd$. Thanks to \eqref{C-S_0} and \eqref{C-S}, we can estimate
\begin{align}
&\bigg|\int_{\Rd}(Q\nabla\zeta_n^2, \nabla |\uu|^2)  |\uu|_\varepsilon^{p'-2}dx\bigg|\notag\\
=&2\bigg|\int_{\Rd}\zeta_n(Q\nabla\zeta_n, \nabla |\uu|^2)  |\uu|_\varepsilon^{p'-2}dx\bigg|\notag\\
\leq &2\varepsilon_1 \int_{\Rd}(Q\nabla |\uu|^2, \nabla |\uu|^2)\zeta_n^2  |\uu|_\varepsilon^{p'-4}dx +\frac{(1+c_0)^2}{2\varepsilon_1}\int_{\Rd}(Q\nabla\zeta_n, \nabla \zeta_n) |\uu|_\varepsilon^{p'}dx,
\label{form-7}
\end{align}

\begin{align}
&\bigg |\int_{\Rd}\sum_{h,k=1}^d (A^{hk}\uu D_k\zeta_n^2,D_h\uu)|\uu|_\varepsilon^{p'-2}dx\bigg |\notag\\
=& 2\bigg |\int_{\Rd}\zeta_n\sum_{h,k=1}^d (A^{hk}\uu D_k\zeta_n,D_h\uu)|\uu|_\varepsilon^{p'-2}dx\bigg |\notag\\
\leq &{\mathscr C}\varepsilon_1\int_{\Rd} \sum_{i=1}^m (Q\nabla u_i, \nabla u_i)\zeta_n^2 |\uu|_\varepsilon^{p'-2} dx +
\frac{4{\mathscr C}}{\varepsilon_1} \int_{\Rd}(Q\nabla \zeta_n, \nabla \zeta_n)  |\uu|_\varepsilon^{p'}dx
\label{form-9}
\end{align}

\begin{align}
& \bigg | \int_{\Rd}\sum_{h,k=1}^d (A^{hk}\uu D_k|\uu|^2 , D_h\uu) \zeta_n^2|\uu|_\varepsilon^{p'-4}dx\bigg |\notag\\
\leq &2{\mathscr C}\varepsilon_2\int_{\Rd} \sum_{i=1}^m (Q\nabla u_i, \nabla u_i) \zeta_n^2|\uu|_\varepsilon^{p'-2} dx +
\frac{{\mathscr C}}{2\varepsilon_2} \int_{\Rd}(Q\nabla |\uu|^2, \nabla |\uu|^2) \zeta_n^2 |\uu|_\varepsilon^{p'-4}dx,
\label{form-99}
\end{align}
for every $\varepsilon_1, \varepsilon_2>0$. Replacing \eqref{form-7}-\eqref{form-99} in the last side of \eqref{form-10}, we get
\begin{align*}
&\lambda \int_{\Rd}\zeta_n^2|\uu|^2|\uu|_\varepsilon^{p'-2} dx\notag\\
\leq
& \left( -1+ {\mathscr C}\varepsilon_1 + {\mathscr C}\varepsilon_2|p'-2|  \right) \int_{\Rd}\sum_{i=1}^m (Q\nabla u_i, \nabla u_i) \zeta_n^2|\uu|_\varepsilon^{p'-2}dx\notag\\
&+\left(-\frac{p'-2}{4} + \varepsilon_1 + \frac{{\mathscr C}|p'-2| }{4\varepsilon_2} \right)\int_{\Rd}(Q\nabla |\uu|^2, \nabla |\uu|^2)  \zeta_n^2|\uu|_\varepsilon^{p'-4}dx\notag\notag\\
&+\left(\frac{(1+c_0)^2}{4\varepsilon_1} + \frac{4{\mathscr C}}{\varepsilon_1} \right)\int_{\Rd}(Q\nabla \zeta_n, \nabla\zeta_n)|\uu|_\varepsilon^{p'}dx=: h_{p'}(\varepsilon_1, \varepsilon_2).
\end{align*}
Setting 
\begin{align*}
A_n=&\int_{\Rd}\sum_{i=1}^m(Q\nabla u_i, \nabla u_i)\zeta_n^2|\uu|_\varepsilon^{p-2} dx,\\[2mm]
B_n=&\int_{\Rd}(Q\nabla|\uu|^2,\nabla|\uu|^2)\zeta_n^2|\uu|_\varepsilon^{p-4} dx,\\[2mm]
C_n=&\int_{\Rd}(Q\nabla\zeta_n, \nabla \zeta_n)  |\uu|_\varepsilon^{p'}dx
\end{align*}
and arguing as in the proof of Lemma \ref{lp_diss} and recalling that $B_n\le 4A_n$, we conclude that
\begin{align*}
\lambda \int_{\Rd}\zeta_n^2|\uu|^2|\uu|_\varepsilon^{p'-2} dx\le &\left( -1 + {\mathscr C}\varepsilon_2|p'-2|\right) A_n
+\left(-\frac{p'-2}{4}+ \frac{{\mathscr C}|p'-2| }{4\varepsilon_2} \right)B_n\\
&+4({\mathscr C}+1)\varepsilon_1 A_n+\left( \frac{(1+c_0)^2}{4\varepsilon_1} + \frac{4{\mathscr C}}{\varepsilon_1} \right)C_n.
\end{align*}

Now, we observe that condition \eqref{condition_p} implies that $p'$ belongs to the set $\widetilde J$ (defined in \eqref{luca_0000}). Therefore, applying the same arguments as in the proof of Lemma \ref{lp_diss}, we conclude that there exists a positive constant $\delta$ such that
\begin{eqnarray*}
\left( -1 + {\mathscr C}\varepsilon_2|p'-2|\right) A_n
+\left(-\frac{p'-2}{4}+ \frac{{\mathscr C}|p'-2| }{4\varepsilon_2} \right)B_n\le -\delta A_n
\end{eqnarray*}
and, using this inequality, we can infer that
\begin{align*}
\lambda \int_{\Rd}\zeta_n^2|\uu|^2|\uu|_\varepsilon^{p'-2} dx\le & [-\delta+4({\mathscr C}+1)\varepsilon_1]A_n
+ \bigg (\frac{(1+c_0)^2}{4\varepsilon_1} + \frac{4{\mathscr C}}{\varepsilon_1}\bigg )C_n.
\end{align*}

Taking $\varepsilon_1=\displaystyle\frac{\delta}{4(\mathscr{C}+1)}$, estimate \eqref{form-5} follows at once.

{\em Step 2.} Here, we complete the proof showing that the 
realization of the operator $\A$ in $L^p(\Rd;\mathbb C^m)$, with maximal domain, generates a strongly continuous semigroup.

First of all, let us observe that Hypotheses \ref{hyp_0} allow to apply the results in Proposition \ref{reg_thm}, in Lemma \ref{lp_diss} and in Step 1 also to the operator \begin{align*}
\A^*=\sum_{h,k=1}^dD_h((Q^{kh})^TD_k)-V^T.
\end{align*}
Now, let $(\overline{\A}, D)$ be the closure of $(\A,C_c^\infty(\Rd;\Rm))$ in $L^p(\Rd;\Rm)$ and fix $\uu \in D$. Then, there exists a sequence $(\uu_n)\in C^\infty_c(\Rd;\Rm)$ such that $\uu_n$ converges to $\uu$ and $\A\uu_n$ converges to $\g=:\overline{\A}\uu$ in $L^p(\Rd;\Rm)$. Moreover, taking the limit as $n$ tends to $+\infty$ in the equality $\langle \uu_n, \A^*\bm \varphi \rangle_{p,p'}=\langle \A \uu_n, \bm \varphi \rangle_{p,p'}$, which holds true for any $\bm \varphi \in C^\infty_c(\Rd;\Rm)$,  we deduce
\begin{equation}\label{int}
\int_{\Rd}(\uu, \A^*\bm \varphi)dx=\int_{\Rd}(\g, \bm \varphi )dx
\end{equation}
for any $\bm \varphi \in C^\infty_c(\Rd;\Rm)$.
The equality \eqref{int} and Proposition \ref{reg_thm}(iv) imply that $\uu\in W^{2,p}_{\rm loc}(\Rd;\Rm)$ and that $\A\uu=\g=\overline{\A}\uu$; hence $\A\uu\in L^p(\Rd; \Rm)$. Consequently, $\uu \in D_{p,\rm{max}}(\A)$.
To prove that $D_{p,\rm{max}}(\A)\subset D$, first we show that  $\lambda I-\A$ is injective on $D_{p,\rm{max}}(\A)$ for some $\lambda>0$. Indeed, let $\uu \in D_{p,\rm{max}}(\A)$ be such that $(\lambda I-\A)\uu=\bm 0$. Then,
\begin{equation}\label{gag}
\int_{\Rd}( \uu, \lambda {\bm \varphi}-\A^*{\bm \varphi}) dx= \int_{\Rd}( \lambda \uu-\A\uu, {\bm \varphi}) dx=0, \qquad\;\, \bm \varphi \in C^\infty_c(\Rd;\Rm).
\end{equation}
Since $C^\infty_c(\Rd;\Rm)$ is a core of the closure of $(\A^*, C^\infty_c(\Rd;\Rm))$ in $L^{p'}(\Rd;\Rm)$, from equality \eqref{gag} we deduce that $\uu\equiv \bm 0$.

Now, we are almost done. Indeed, fix $\uu \in D_{p,\rm{max}}(\A)$ and set $\vv= \lambda \uu-\A\uu$. Then, $\vv \in L^p(\Rd;\Rm)$ and Step 1
guarantees the existence of a function $\zz\in D$ such that $\lambda \zz-\overline{\A}\zz=\vv= \lambda \uu-\A\uu$.
Since $D\subset D_{p,\rm{max}}(\A)$, the function $\ww=\zz-\uu$ belongs to $D_{p,\rm{max}}(\A)$ and satisfies the equation $\lambda \ww-\A\ww=\bm 0$. The injectivity of $\lambda-\A$ on $D_{p,\rm{max}}(\A)$ yields immediately that $\ww=\bm 0$ or equivalently that $\uu=\zz\in D$. The last assertion of the claim then easily follows by the equality $D= D_{p,\rm{max}}(\A)$.
\end{proof}

\begin{rmk}\label{esami}
{\rm 
\begin{enumerate}[\rm (i)]
\item
If we take $v_{ij}\equiv 0$ for any $i,j=1,\ldots, m$ in Theorem \ref{mainth} then we deduce that the realization ${\bf A}_{0,p}$ of $\A_0=\sum_{h,k=1}^d D_h(Q^{hk}D_k)$ in $L^p(\Rd;\Rm)$ endowed with the maximal domain $D_{p,\rm{max}}(\A_0)$ generates a contractive semigroup in $L^p(\Rd;\R^m)$. As a consequence, $({\bf A}_{0,p},D_{p,\rm{max}}(\A_0))$ is a closed operator in $L^p(\Rd;\Rm)$.
\item
Condition \eqref{ipo_core} has been already considered in \cite{ALM} in the scalar case and in the context of $L^p$-spaces related to invariant measures.
\item
Condition \eqref{condition_p} is the best condition on $p$ which guarantees that both $p$ and $p'$ belong to $\widetilde J$.
\end{enumerate}}
\end{rmk}

\begin{thm}\label{sect_pro}
Besides the assumptions of Theorem $\ref{mainth}$,  assume  that
\begin{equation}\label{ipo_ana}
|{\rm Im}(V\zeta,\zeta)| \le c_V {\rm Re}(V \zeta, \zeta)
\end{equation}
in $\Rd$, for every $\zeta \in \Cm$ and some positive constant $c_V$. Then, for every $p$ which satisfies \eqref{condition_p}, the operator ${\bf A}_p$ generates an analytic semigroup in $L^p(\Rd;\mathbb C^m)$.
\end{thm}

\begin{proof}
By \cite[Chapter I, Section 5.8]{goldstein} and taking into account that $C^{\infty}_c(\Rd;\mathbb C^m)$ is a core of ${\bf A}_p$, it suffices to show that there exists a positive constant $C_p$ such that 
\begin{equation*}
\left|{\rm Im} \int_{\Rd}({\A}\uu,\uu)|\uu|^{p-2} dx\right| \leq  -C_p{\rm Re} \int_{\Rd}(\A\uu,\uu)|\uu|^{p-2} dx
\end{equation*}
for every $\uu\in C_c^\infty(\Rd, \Cm)$.

First of all, we point out that, thanks to \eqref{C-S_0} we can estimate
\begin{align}\label{C-S1}
\bigg |\sum_{i=1}^m (Q\vartheta_i,\eta_i)\bigg |\le &
\sum_{i=1}^m |(Q\vartheta_i,\eta_i)|\notag\\
\le &(1+c_0)\sum_{i=1}^m ({\rm Re}(Q\vartheta_i,\vartheta_i))^{\frac{1}{2}}({\rm Re}(Q\eta_i,\eta_i))^{\frac{1}{2}}\notag\\
\le &(1+c_0)\bigg ({\rm Re}\sum_{i=1}^m(Q\vartheta_i,\vartheta_i)\bigg )^{\frac{1}{2}}\bigg ({\rm Re}\sum_{i=1}^m(Q\eta_i,\eta_i)\bigg )^{\frac{1}{2}}
\end{align}
for any $\vartheta_i, \eta_i\in\mathbb{C}^d$, $(i=1, \ldots, m)$.
Moreover, from Hypothesis \ref{hyp_0}(iii) and the formula $(V\uu,\uu)=(\A_0\uu,\uu)-(\A\uu,\uu)$, it follows that
\begin{align*}
0 \le &{\rm Re}\int_{\Rd}\bigg(\sum_{i=1}^m (Q \nabla u_i, \nabla u_i)+(V\uu,\uu)\bigg )  |\uu|_\varepsilon^{p-2}dx\\
= &
{\rm Re}\int_{\Rd}\bigg(\sum_{i=1}^m (Q \nabla u_i, \nabla u_i)+(\A_0\uu,\uu)-(\A\uu,\uu)\bigg )  |\uu|_\varepsilon^{p-2}dx\notag\\
\le &
{\rm Re}\int_{\Rd}\bigg(\sum_{i=1}^m (Q \nabla u_i, \nabla u_i)-(\A\uu,\uu)\bigg )  |\uu|_\varepsilon^{p-2}dx\notag
\end{align*}
for any $\uu\in C^{\infty}_c(\Rd;\Cm)$,
where we have used Lemma \ref{lp_diss} to deduce that
\begin{eqnarray*}
{\rm Re}\int_{\Rd}(\A_0\uu,\uu)|\uu|_{\varepsilon}^{p-2}dx\le 0.
\end{eqnarray*}
Hence, taking advantage of formula \eqref{luca_00}, we can infer that
for any $p$, which satisfies condition \eqref{condition_p}, there exists a positive constant $c_p$ such that
\begin{equation}
{\rm Re}\int_{\Rd}\bigg(\sum_{i=1}^m (Q \nabla u_i, \nabla u_i)+(V\uu,\uu)\bigg )  |\uu|_\varepsilon^{p-2}dx\le -c_p\int_{\Rd}{\rm Re}(\A\uu,\uu)|\uu|_\varepsilon^{p-2}dx
\label{prelim}
\end{equation}
for all $\uu\in C_c^\infty(\Rd, \Cm)$.
Thus, taking formula \eqref{formula} into account and using the assumptions \eqref{Im_Q}, \eqref{comLeg} and also condition \eqref{ipo_ana}, we obtain
\begin{align*}
\bigg |{\rm Im}\int_{\Rd}(\A\uu,\uu)|\uu|_\varepsilon^{p-2} dx\bigg |
\le & \int_{\Rd}\sum_{i=1}^m|{\rm Im}(Q \nabla u_i, \nabla u_i)||\uu|_\varepsilon^{p-2}dx\\
&+\int_{\Rd}\bigg |{\rm Im}\sum_{h,k=1}^d (A^{hk}D_k \uu, D_h\uu)\bigg | |\uu|_\varepsilon^{p-2}dx\notag\\
&+ \frac{|p-2|}{2}\int_{\Rd}\bigg |{\rm Im}\sum_{i=1}^m(Q\nabla u_i,\nabla|\uu|^2)\overline{u_i}\bigg ||\uu|_\varepsilon^{p-4}dx\\
&+ \frac{|p-2|}{2} \int_{\Rd}\bigg |{\rm Im}\sum_{h,k=1}^d (A^{hk}D_k \uu, \uu)D_h|\uu|^2\bigg | |\uu|_\varepsilon^{p-4}dx\\
&+\int_{\Rd}|{\rm Im}(V\uu,\uu)|
|\uu|_\varepsilon^{p-2}dx\notag\\
\le & (c_0+\mathscr{C})\int_{\Rd}\sum_{i=1}^m{\rm Re}(Q \nabla u_i, \nabla u_i)|\uu|_\varepsilon^{p-2}dx\\
&+ \frac{|p-2|}{2}\int_{\Rd}\bigg |{\rm Im}\sum_{i=1}^m(Q\nabla u_i,\nabla|\uu|^2)\overline{u_i}\bigg ||\uu|_\varepsilon^{p-4}dx\\
&+ \frac{|p-2|}{2} \int_{\Rd}\bigg |{\rm Im}\sum_{h,k=1}^d (A^{hk}D_k \uu, \uu)D_h|\uu|^2\bigg | |\uu|_\varepsilon^{p-4}dx\\
&+c_V\int_{\Rd}{\rm Re}(V\uu,\uu)
|\uu|_\varepsilon^{p-2}dx.
\end{align*}

Now, using the Cauchy-Schwarz inequalities \eqref{C-S} and \eqref{C-S1},  we can estimate
\begin{align*}
&\bigg |{\rm Im}\sum_{i=1}^m(Q\nabla u_i,\nabla |\uu|^2)\overline{u_i}\bigg ||\uu|_\varepsilon^{p-4}+
\bigg |{\rm Im}\sum_{h,k=1}^d(A^{hk}D_k \uu, \uu D_h|\uu|^2)\Big||\uu|_\varepsilon^{p-4}\\
\le & (1+c_0+2\mathscr{C})|\uu|\bigg ({\rm Re}\sum_{i=1}^m(Q\nabla u_i,\nabla u_i)\bigg )^{\frac{1}{2}}(Q\nabla |\uu|^2,\nabla |\uu|^2)^{\frac{1}{2}}|\uu|_\varepsilon^{p-4}\\
\le & (1+c_0+2\mathscr{C})\bigg ({\rm Re}\sum_{i=1}^m(Q\nabla u_i,\nabla u_i)\bigg )^{\frac{1}{2}}(Q\nabla |\uu|^2,\nabla |\uu|^2)^{\frac{1}{2}}|\uu|_\varepsilon^{p-3}\\
\le &\bigg (\frac{1+c_0}{2}+\mathscr{C}\bigg )\left( {\rm Re}\sum_{i=1}^m(Q\nabla u_i,\nabla u_i) |\uu|_\varepsilon^{p-2}+  (Q \nabla |\uu|^2,\nabla|\uu|^2)|\uu|_\varepsilon^{p-4}\right)\\
\le & 5\bigg (\frac{1+c_0}{2}+\mathscr{C}\bigg ) {\rm Re}\sum_{i=1}^m(Q\nabla u_i,\nabla u_i)|\uu|_\varepsilon^{p-2},
\end{align*}
where in the last line we used estimate \eqref{disDmod}.
Summing up, we obtain
\begin{align*}
\bigg |{\rm Im}\int_{\Rd}(\A\uu,\uu)|\uu|_\varepsilon^{p-2} dx\bigg |\le & C_1\int_{\Rd} {\rm Re}\sum_{i=1}^m(Q\nabla u_i,\nabla u_i) |\uu|_\varepsilon^{p-2}dx\\
&+ c_V \int_{\Rd}{\rm Re}(V\uu,\uu)|\uu|_{\varepsilon}^{p-2}dx,
\end{align*}
where $C_1$ is a positive constant depending only on $c_0$, $\mathscr{C}$ and $p$. Using estimate \eqref{prelim} and letting $\varepsilon$ tend to $0$, we conclude the proof.
\end{proof}

\section{Domain characterization}
\label{sect-4}
In this section, under additional conditions on the matrix-valued functions $Q$ and $V$ we provide a characterization of the domain of ${\bf A}_p$. We start by a preliminary result.

\begin{prop}\label{prop_vu}
Under Hypothesis $\ref{hyp_0}$, assume that there exist a function $v\in C^1(\Rd)$, with positive infimum $v_0$, and two positive constants $\gamma$ and $C_\gamma$ such that
\begin{equation}\label{ipov}
(i)\,\,\,(Q \nabla v, \nabla v)^{\frac{1}{2}} \le \gamma v^{\frac{3}{2}}+C_\gamma,\qquad\;\,(ii)\,\,\,(V\xi,\xi)\ge v|\xi|^2
\end{equation}
in $\Rd$ for any $\xi \in \R^m$. Further, assume that $\mathscr C\in (0,\frac{1}{2})$, $p\in \left (1+\frac{6\mathscr{C}}{4{\mathscr C}+1},\frac{3}{2}+\frac{1}{4\mathscr C}\right )$ and
\begin{equation}\label{ipo_p}
\Lambda_p:=1-\frac{(p-1)^2(\gamma v_0^{3/2}+C_\gamma )^2(1+c_0+2\mathscr{C})^2}{4v_0^3\Theta_{p,\mathscr{C}}}>0,
\end{equation}
where
\begin{equation*}
\Theta_{p,\mathscr{C}}=\left\{
\begin{array}{ll}
p-1-2\mathscr{C}(5-2p), \quad &p \in (1,2),\\[1mm]
1-2\mathscr{C}(2p-3), \quad &p\in [2,\infty).
\end{array}
\right.
\end{equation*}
Then, there exists a positive constant $K$ such that
\begin{equation}\label{vu}
\|v\uu\|_p \le K\|\A\uu\|_p, \qquad\;\, \uu\in C^\infty_c(\Rd;\mathbb{C}^m).
\end{equation}
\end{prop}
\begin{proof}
Since the coefficients of the operator $\A$ are real-valued, we can limit ourselves to considering functions with values in $\R^m$. We fix $\uu\in C^{\infty}_c(\Rd;\Rm)$, $\varepsilon>0$ and set $\f=-\A\uu$. Then,
\begin{align}
\int_{\Rd}(\f,\uu)|\uu|_\varepsilon^{p-2}v^{p-1}dx=&-\int_{\Rd}(\A\uu,\uu)|\uu|_\varepsilon^{p-2}v^{p-1}dx\notag\\
=&\sum_{i=1}^m \int_{\Rd}(Q \nabla u_i,\nabla u_i)|\uu|_{\varepsilon}^{p-2}v^{p-1}dx\notag\\
&+\frac{p-2}{4} \int_{\Rd}(Q\nabla |\uu|^2,\nabla |\uu|^2)
|\uu|_{\varepsilon}^{p-4}v^{p-1}dx\notag\\
&+\frac{p-1}{2} \int_{\Rd}(Q\nabla |\uu|^2,\nabla v)|\uu|_{\varepsilon}^{p-2}v^{p-2} dx\notag\\
&+\sum_{h,k=1}^d \int_{\Rd}(A^{hk}D_k \uu, D_h \uu)|\uu|_{\varepsilon}^{p-2}v^{p-1}dx\notag\\
&+\frac{p-2}{2}\sum_{h, k=1}^d \int_{\Rd}(A^{hk}D_k \uu, \uu)(D_h|\uu|^2)|\uu|_{\varepsilon}^{p-4}v^{p-1}dx\notag\\
&+(p-1)\sum_{h, k=1}^d \int_{\Rd}(A^{hk}D_k \uu,\uu)|\uu|_{\varepsilon}^{p-2}v^{p-2}D_h v dx\notag\\
&+\int_{\Rd}(V\uu,\uu)|\uu|_\varepsilon^{p-2}v^{p-1}dx.
\label{form-11}
\end{align}
We denote by $\Gamma_i$ $(i=1,\ldots.5)$ the last five terms in the right-hand side of \eqref{form-11} and estimate them.
Let us start from $\Gamma_1$ that, thanks to the Cauchy-Schwartz inequality \eqref{C-S_0}, the inequality \eqref{disDmod} and condition \eqref{ipov}(i), can be estimated as follows:
\begin{align*}
\Gamma_1 
\ge& -\frac{p-1}{2}(1+c_0)  \int_{\Rd}(Q \nabla|\uu|^2, \nabla |\uu|^2) ^{\frac{1}{2}}(Q \nabla v, \nabla v)^{\frac{1}{2}}|\uu|_\varepsilon^{p-2}v^{p-2}dx\\
\ge& - (p-1)(1+c_0)\gamma  \int_{\Rd}\bigg (\sum_{i=1}^m(Q \nabla u_i, \nabla u_i)\bigg )^{\frac{1}{2}}|\uu||\uu|_\varepsilon^{p-2}v^{p-\frac{1}{2}}dx\\
& - (p-1)(1+c_0)C_\gamma  \int_{\Rd}\bigg (\sum_{i=1}^m(Q \nabla u_i, \nabla u_i)\bigg )^{\frac{1}{2}}|\uu||\uu|_\varepsilon^{p-2}v^{p-2}dx\\
\ge &-(p-1)(1+c_0)\gamma\left(\int_{\Rd}v^p|\uu|^2|\uu|_\varepsilon^{p-2} dx\right)^{\frac{1}{2}}\bigg (\int_{\Rd} \sum_{i=1}^m(Q \nabla u_i, \nabla u_i)|\uu|_\varepsilon^{p-2}v^{p-1} dx\bigg )^{\frac{1}{2}}\\
&- (p-1)(1+c_0)C_\gamma  \left(\int_{\Rd}v^{p-3}|\uu|^2|\uu|_\varepsilon^{p-2} dx\right)^{\frac{1}{2}}\\
&\qquad\qquad\quad\;\,\times\bigg (\int_{\Rd} \sum_{i=1}^m(Q \nabla u_i, \nabla u_i)|\uu|_\varepsilon^{p-2}v^{p-1} dx\bigg )^{\frac{1}{2}}\\
\ge &-(p-1)(1+c_0)\left( \frac{\gamma}{4\varepsilon_0}+\frac{C_\gamma}{4v_0^{3}\varepsilon_1}\right)\int_{\Rd}v^p|\uu|^2|\uu|_\varepsilon^{p-2}dx\\
&- (p-1)(1+c_0)(\gamma\varepsilon_0+C_\gamma\varepsilon_1) \int_{\Rd}\sum_{i=1}^m(Q \nabla u_i, \nabla u_i)|\uu|_\varepsilon^{p-2}v^{p-1}dx
\end{align*}
for every $\varepsilon_0, \varepsilon_1>0$. Moreover, using the Cauchy-Schwarz inequality \eqref{C-S} and the inequality $|\uu|\le|\uu|_{\varepsilon}$, we can estimate 
\begin{align*}
\Gamma_2 
\ge& -\sum_{h, k=1}^d \int_{\Rd}|(A^{hk}D_k \uu, D_h \uu)||\uu|_{\varepsilon}^{p-2}v^{p-1}dx\\
\ge&  - 2\mathscr{C} \int_{\Rd}\sum_{i=1}^m(Q \nabla u_i, \nabla u_i)|\uu|_\varepsilon^{p-2}v^{p-1}dx\\[2mm]
\Gamma_3 
\ge &-\frac{|p-2|}{2}\sum_{h,k=1}^d\int_{\Rd}|(A^{hk}D_k \uu, \uu D_h |\uu|^2)||\uu|_{\varepsilon}^{p-4}v^{p-1}dx\\
\ge& -2\mathscr{C}|p-2|\int_{\Rd}\bigg (\sum_{i=1}^m (Q \nabla u_i, \nabla u_i)\bigg )^{\frac{1}{2}}(Q \nabla |\uu|^2 , \nabla |\uu|^2)^{\frac{1}{2}}|\uu||\uu|_{\varepsilon}^{p-4}v^{p-1}dx\\
\ge & -4\mathscr{C}|p-2|\int_{\Rd}\sum_{i=1}^m(Q \nabla u_i, \nabla u_i)|\uu|_{\varepsilon}^{p-2}v^{p-1}dx.
\end{align*}
Further, arguing as in the estimate of $\Gamma_1$, we get
\begin{align*}
\Gamma_4
\ge& -(p-1) \int_{\Rd}\Big|\sum_{h, k=1}^d(A^{hk}D_k\uu, \uu D_h v)\Big||\uu|_{\varepsilon}^{p-2}v^{p-2}dx\\
\ge & -2\mathscr{C}(p-1)\int_{\Rd}\bigg (\sum_{i=1}^m( Q\nabla u_i,\nabla u_i) \bigg)^{\frac{1}{2}}( Q \nabla v, \nabla v)^{\frac{1}{2}}|\uu||\uu|_\varepsilon^{p-2}v^{p-2}dx\\
\ge& -2\mathscr{C}(p-1)\gamma\int_{\Rd}\bigg (\sum_{i=1}^m(Q\nabla u_i,\nabla u_i) \bigg )^{\frac{1}{2}}|\uu||\uu|_\varepsilon^{p-2}v^{p-\frac{1}{2}}dx\\
&-2\mathscr{C}C_\gamma(p-1)\int_{\Rd}\bigg (\sum_{i=1}^m( Q\nabla u_i,\nabla u_i) \bigg )^{\frac{1}{2}}|\uu||\uu|_\varepsilon^{p-2}v^{p-2}dx\\
\ge &-\mathscr{C}(p-1)(\gamma \varepsilon_2+C_\gamma \varepsilon_3)\int_{\Rd}\sum_{i=1}^m(Q\nabla u_i,\nabla u_i)|\uu|_\varepsilon^{p-2}v^{p-1}dx\\
&-\mathscr{C}(p-1)\left(\frac{\gamma}{\varepsilon_2}+\frac{C_\gamma}{v_0^3\varepsilon_3}\right)\int_{\Rd}|\uu|^2|\uu|_\varepsilon^{p-2} v^pdx
\end{align*}
for every $\varepsilon_2, \varepsilon_3>0$.
Finally, using \eqref{ipov}(ii) it follows that
\begin{align*}
\Gamma_5=\int_{\Rd}(V\uu,\uu)|\uu|_\varepsilon^{p-2}v^{p-1}dx \ge \int_{\Rd}|\uu|^2|\uu|_\varepsilon^{p-2}v^pdx.
\end{align*}
Summing up, we have proved that
\begin{align}
\int_{\Rd}(\f,\uu)|\uu|_\varepsilon^{p-2}v^{p-1}dx \ge & \,\,\psi_1\sum_{i=1}^m \int_{\Rd}(Q \nabla u_i,\nabla u_i)|\uu|_{\varepsilon}^{p-2}v^{p-1}dx\notag\\
&+\frac{p-2}{4} \int_{\Rd}(Q\nabla|\uu|^2,\nabla|\uu|^2)|\uu|_{\varepsilon}^{p-4}v^{p-1}dx\notag\\
&+\psi_2\int_{\Rd}|\uu|^2|\uu|_\varepsilon^{p-2}v^pdx,
\label{previous}
\end{align}
where 
\begin{align*}\psi_1=&\psi_1(\varepsilon_0, \varepsilon_1, \varepsilon_2, \varepsilon_3)\\
=&\left[1-(p-1)(1+c_0)(\gamma\varepsilon_0+C_\gamma \varepsilon_1)-2\mathscr{C}-4\mathscr{C}|p-2|-\mathscr{C}(p-1)(\gamma \varepsilon_2+C_\gamma \varepsilon_3)\right]
\end{align*}
and
\begin{align}\label{f2}
\psi_2=&\psi_2(\varepsilon_0, \varepsilon_1, \varepsilon_2, \varepsilon_3)\notag\\
=&\left[1-(p-1)(1+c_0)\left( \frac{\gamma}{4\varepsilon_0}+\frac{C_\gamma}{4v_0^3\varepsilon_1}\right)-\mathscr{C}(p-1)\left(\frac{\gamma}{\varepsilon_2}+\frac{C_\gamma}{v_0^3\varepsilon_3}\right)\right].
\end{align}
Thus, combining \eqref{previous} with the estimate
\begin{eqnarray*}
\int_{\Rd}(\f,\uu)|\uu|_\varepsilon^{p-2}v^{p-1}dx \le 
\delta \int_{B(0,R)} |\uu|_\varepsilon^p v^pdx+ C(\delta, p)\int_{\Rd}|\f|^p dx,
\end{eqnarray*}
which holds true for any $\delta>0$ and some positive constant $C(\delta, p)$, where $R>0$ is such that ${\rm supp}(\f)\subset B(0,R)$, we deduce
\begin{align}\label{pre}
&\psi_1\sum_{i=1}^m \int_{\Rd}(Q \nabla u_i,\nabla u_i)|\uu|_{\varepsilon}^{p-2}v^{p-1}dx\notag\\
&+\frac{p-2}{4} \int_{\Rd}(Q \nabla|\uu|^2,\nabla|\uu|^2\rangle|\uu|_{\varepsilon}^{p-4}v^{p-1}dx\notag\\
&+\psi_2\int_{\Rd}|\uu|^2|\uu|_\varepsilon^{p-2}v^pdx - \delta \int_{B(0,R)} |\uu|_\varepsilon^p v^pdx\le  C(\delta, p)\int_{\Rd}|\f|^p dx.
\end{align}

Now, we distinguish the cases $p\ge 2$ and $p \in (1,2)$. In the first case, neglecting the second term in the left-hand side of the previous inequality, we deduce that
\begin{align*}
&\psi_1\sum_{i=1}^m \int_{\Rd}(Q \nabla u_i,\nabla u_i)|\uu|_{\varepsilon}^{p-2}v^{p-1}dx\notag\\
&+\psi_2\int_{\Rd}|\uu|^2|\uu|_\varepsilon^{p-2}v^pdx - \delta \int_{B(0,R)} |\uu|_\varepsilon^pv^p dx\le  C(\delta, p)\int_{\Rd}|\f|^p dx.
\end{align*}
On the other hand, if $p<2$ then starting from \eqref{pre} and using \eqref{disDmod} we deduce that
\begin{align*}
&(\psi_1+p-2)\sum_{i=1}^m \int_{\Rd}(Q \nabla u_i,\nabla u_i)|\uu|_{\varepsilon}^{p-2}v^{p-1}dx\notag\\
&+\psi_2\int_{\Rd}v^p|\uu|^2|\uu|_\varepsilon^{p-2}dx - \delta \int_{B(0,R)} v^p|\uu|_\varepsilon^p dx\le  C(\delta, p)\int_{\Rd}|\f|^p dx.
\end{align*}
Now, computing the supremum of the function $\psi_2$ in the set $\Omega=\{(\varepsilon_0, \varepsilon_1, \varepsilon_2, \varepsilon_3)\in (0,+\infty)^4: \psi_1(\varepsilon_0, \varepsilon_1, \varepsilon_2, \varepsilon_3)\ge 0\}$, if $p \ge 2$, and on the set 
$\tilde\Omega=\{(\varepsilon_0, \varepsilon_1, \varepsilon_2, \varepsilon_3)\in (0,+\infty)^4: \psi_1(\varepsilon_0, \varepsilon_1, \varepsilon_2, \varepsilon_3)+p-2 \ge 0\}$, if $p\in (1,2)$, we get
\begin{equation*}\Lambda_p\int_{\Rd}v^p|\uu|^2|\uu|_\varepsilon^{p-2}dx - \delta \int_{B(0,R)} v^p|\uu|_\varepsilon^p dx\le  C(\delta, p)\int_{\Rd}|\f|^p dx
\end{equation*}
(We refer the reader to Appendix \ref{app-B} for further details). Due to the positivity of $\Lambda_p$ (see its definition in \eqref{ipo_p}), letting $\varepsilon$ tend to $0$ and choosing $\delta$ small enough, we get \eqref{vu}, completing the proof.
\end{proof}

Now, invoking \cite[Theorem 5.9]{goldstein} we prove our main generation result which provides also a domain characterization. We recall that $\A_0$ is the operator defined in Remark \ref{esami}(i) by
$\A_0=\sum_{h,k=1}^dD_h(Q^{hk}D_k)$ and $\A_{0,p}$ is its realization in $L^p(\Rd;\mathbb C^m)$.

\begin{thm}
Let Hypotheses $\ref{hyp_0}$ be satisfied with  $\mathscr C\in (0,1/2)$. Further, assume that conditions \eqref{ipo_core}, \eqref{ipo_ana} and \eqref{ipov} are satisfied and that there exists a positive constant $c_1$ such that $|V(x)\xi|\le c_1 v(x)|\xi|$ for every $x \in \Rd$ and $\xi \in \Rm$. Then, 
for every 
$p\in \left (1+\frac{6\mathscr C}{4\mathscr C+1},\frac{3}{2}+\frac{1}{4\mathscr C}\right )$, which satisfies \eqref{ipo_p}, 
it holds that 
\begin{eqnarray*}
D_{p,\rm{max}}(\A)=\{\uu \in W^{2,p}_{\rm loc}(\Rd;\Cm): v\uu, \A_0\uu \in L^p(\Rd; \Cm)\}=:D_p.
\end{eqnarray*}
Consequently, $({\bf A}_p, D_p)$ generates an analytic contraction semigroup in $L^p(\Rd;\Cm)$.
\end{thm}

\begin{proof} 
First of all let us observe that all the assumptions in Theorems \ref{mainth}, \ref{sect_pro} and Proposition \ref{prop_vu} are satisfied; hence all the results therein hold true.

Now, let us fix $p \in (1,\infty)$ and observe that, if we endow $D_p$ with the norm $\|\uu\|_{D_p}:=\|v\uu\|_{L^p(\Rd;\Cm)}+\|\A_0 \uu\|_{L^p(\Rd;\Cm)}$, then $(D_p, \|\cdot\|_{D_p})$ is a Banach space.
Indeed, let $(\uu_n)$ be a Cauchy sequence in $D_p$. Since the infimum over $\Rd$ of $v$ is strictly positive, $(\uu_n)$ and $(\A_0\uu_n)$ are Cauchy sequences in $L^p(\Rd;\Cm)$. In addition, since $D_p\subset D_{p,\max}(\A_0)$ and $(\A_{0,p}, D_{p,\max}(\A_0))$ is closed in $L^p(\Rd;\Cm)$ (see Remark \ref{esami}(i)), it follows that there exists $\uu \in D_{p,\rm{max}}(\A_0)$ such that $\uu_n$ and $\A_0 \uu_n$ converge to $\uu$ and $\A_0 \uu$ in $L^p(\Rd;\Cm)$, respectively. The closedness of the multiplication operator $\uu\mapsto v\uu$ allows us to conclude that $v\uu_n$ converges to $v\uu$ in $L^p(\Rd;\Cm)$ and, consequently, that $\uu$ belongs to $D_p$.

Denote by $\A_p$ the realization of $\A$ in $L^p(\Rd;\mathbb{C}^m)$ with domain $D_p$. In view of \cite[Theorem 5.9]{goldstein} and since $C^{\infty}_c(\Rd,\Cm)$ is dense in $L^p(\Rd;\Cm)$, we only need to prove that
$1 \in \rho(\A_p)$.
To complete the proof, we first show that there exist two positive  constants $M_1, M_2$ such that
\begin{equation}\label{est_li}
M_1\|\uu\|_{D_p} \le \|\A\uu-\uu\|_{L^p(\Rd;\Cm)}\le M_2\|\uu\|_{D_p}, \qquad\;\, \uu \in C^\infty_c(\Rd;\Cm).
\end{equation}
To prove \eqref{est_li} we can limit ourselves to considering functions which take values in $\Rm$. Fix $\uu\in C^{\infty}_c(\Rd;\Rm)$. Using the assumption on $V$, we can estimate $\|\A_0\uu\|_p\le\|\A\uu\|_p+\|V\uu\|_p\le
\|\A\uu\|_p+c_1\|v\uu\|_p$. Hence, thanks to \eqref{vu} and the previous estimate, we get
\begin{align*}
\|\uu\|_{D_p}=&\|v\uu\|_p+\|\A_0\uu\|_p
\le[1+(c_1+1)K]\|\A\uu\|_p.
\end{align*}
Thus, using the $L^p$-dissipativity of the operator $\A$ (see Lemma \ref{lp_diss}), which implies that $\|\uu\|_p \le \|\uu-\A\uu\|_p$ for any $\uu \in C^\infty_c(\Rd;\R^m)$, we deduce that
\begin{eqnarray*}
\|\uu\|_{D_p} \le [1+(c_1+1)K](\|\A\uu-\uu\|_p+ \|\uu\|_p)\le 2[1+(c_1+1)K]\|\A\uu-\uu\|_p.
\end{eqnarray*}
Hence, the first inequality in \eqref{est_li} follows with $ M_1= [2(1+K+c_1K)]^{-1}$. On the other hand, using again the assumption on $V$, we get
\begin{align*}
\|\A\uu-\uu\|_p\le & \|\A_0\uu\|_p+\|V\uu\|_p+\|\uu\|_p\\
\le & \|\A_0\uu\|_p+c_1\|v\uu\|_p+ v_0^{-1}\|v\uu\|_p\\
\le &(1+c_1+v_0^{-1} )\|\uu\|_{D_p},
\end{align*}
where $v_0$ is the positive infimum over $\Rd$ of $v$.
Hence, the second inequality in \eqref{est_li} holds true with $M_2=1+c_1+v_0^{-1}$.

Using \eqref{est_li} we can now prove that
\begin{eqnarray*}
D_p= D_{p,\rm{max}}(\A)=\{\uu \in W^{2,p}_{\rm loc}(\Rd;\Cm)\cap L^p(\Rd;\Cm): \, \A\uu \in L^p(\Rd;\Cm)\}.
\end{eqnarray*}

Clearly, $D_p$ is contained in $D_{p,\rm{max}}(\A)$. To prove the other inclusion, let us fix $\uu \in D_{p,\rm{max}}(\A)$. Since $C^\infty_c(\Rd;\Cm)$ is a core of $({\bf A}_p,D_{p,\rm{max}}(\A))$, there exists a sequence $(\uu_n)\subset C^\infty_c(\Rd;\Cm)$ such that $\uu_n$ converges to $\uu$ and $\A\uu_n$ converges to $\A\uu$ in $L^p(\Rd;\Cm)$, as $n$ tends to $\infty$. From \eqref{est_li} we deduce that $(\uu_n)$ is a Cauchy sequence in $(D_p, \|\cdot\|_{D_p})$ and, since  $(D_p, \|\cdot\|_{D_p})$ is a Banach space, we conclude that $\uu$ belongs to $D_p$ showing that $D_{p,\rm{max}}(\A)\subset D_p$. Now, Theorem \ref{mainth} yields the claim. 
\end{proof}

\section{Examples}
\label{sect-5}
In this section we provide classes of operators to which our results can be applied. Recall that the matrix-valued functions $Q^{hk}$  ($h,k=1, \ldots,d$) can be written as $q_{hk}I+A^{hk}$ where $q_{hk}:\Rd \to \R$ and $A^{hk}$ are $m\times m$ matrix-valued functions. In the following examples, we assume that $q_{hk}, a^{hk}_{ij} \in C^1(\Rd)$, $v_{ij} \in L^\infty_{\rm loc}(\Rd)$, for any $i,j=1, \ldots, m$, $h,k=1, \ldots,d$, and that ${\rm Re}(V(x)\zeta,\zeta)\ge 0$ for any $x \in \Rd$ and $\zeta \in \mathbb{C}^m$.

We also recall that the main assumptions required on the $d\times d$ matrix-valued functions $Q=(q_{hk})_{h, k=1}^d$ are the following ones:
\begin{equation}\label{cond}
(i)\,\,{\rm Re}\,(Q(x)\xi, \xi) >0,\quad \quad\quad (ii)\,\,|({\rm Im}(Q(x)\xi,\xi)|\le c_0{\rm Re}(Q(x)\xi,\xi)
\end{equation}
for any $x \in \Rd$, $\xi \in  \mathbb C^d\setminus\{0\}$ and some positive constant $c_0$.
Conditions \eqref{cond} are satisfied, for instance, when
\begin{enumerate}[\rm(1)]
\item $Q(x)$ is a positive definite and symmetric real-valued matrix for any $x \in \Rd$ and $\inf_{x \in \Rd}\lambda_Q(x)>0$, where $\lambda_Q(x)$ denotes the minimum eigenvalue of $Q(x)$;
\item $Q(x)$ is a diagonal perturbation of an antisymmetric matrix-valued function for any $x \in \Rd$, i.e., 
\begin{eqnarray*}
Q(x)={\rm diag}(q_{11}(x), \ldots, q_{dd}(x))+ Q_0(x),\qquad\;\, x \in \Rd,
\end{eqnarray*}
where $Q_0(x)$ is an antisymmetric matrix for any $x \in \Rd$, and there exist positive constants $k_1,k_2$ such that
\begin{equation*}
\inf_{x \in \Rd}q_{ii}(x)>k_1,\quad\;\, i=1, \ldots,d,
\end{equation*}
and
\begin{equation*}
\sum_{j\in\{1,\ldots,d\}\setminus\{i\}}|q_{ij}(x)|\le  k_2q_{ii}(x),\qquad\;\,  i\in \{1, \ldots, d\},\,\,  x \in \Rd.
\end{equation*}
\end{enumerate}

In the latter case
\begin{equation}\label{real_part}
{\rm Re}(Q(x)\zeta,\zeta)= \sum_{i=1}^dq_{ii}(x)|\zeta_i|^2>k_1|\zeta|^2
\end{equation}
and
\begin{equation*}
{\rm Im} (Q(x)\zeta,\zeta)=\sum_{i=1}^d\sum_{j\in\{1,\ldots,d\}\setminus\{i\}}q_{ij}(x)\zeta_i\overline{\zeta}_j
\end{equation*}
for every $x\in\Rd$ and $\zeta=(\zeta_1,\ldots,\zeta_d)\in\mathbb C^d$.
Thus, it follows that
\begin{align*}
|{\rm Im} (Q(x)\zeta,\zeta)|
& \le \frac{1}{2}\sum_{i=1}^d\sum_{j\in\{1,\ldots,d\}\setminus\{i\}}|q_{ij}(x)|(|\zeta_i|^2+|\zeta_j|^2)\\
&= \sum_{i=1}^d|\zeta_i|^2\sum_{j\in\{1,\ldots,d\}\setminus\{i\}}|q_{ij}(x)|\\
&\le k_2\re (Q(x)\zeta,\zeta)
\end{align*}
for every $\zeta=(\zeta_1,\ldots,\zeta_d)\in\mathbb C^d$ and $x\in\Rd$,
whence condition \eqref{cond}(ii) is satisfied with $c_0=k_2$.

\subsection{The symmetric case I}
Here, we assume that $Q$ is as described in (1). Clearly, condition \eqref{cond}(i) is trivially satisfied as well as condition \eqref{cond}(ii) holds true with $c_0=0$.
Concerning the matrices $A^{hk}$, assume that 
\begin{equation}\label{cond_i}\sum_{h,k=1}^d(A^{hk}(x) \vartheta^k,\vartheta^h)\ge 0,\qquad\;\, x \in \Rd,
\end{equation}
for any $\vartheta^k, \vartheta^h\in \R^m$
 and that there exists a positive constant $k_0$ such that
\begin{equation}\label{cond_ii}
|a^{hk}_{ij}(x)|\le k_0 \lambda_Q(x), \qquad \;\, x\in \Rd,
\end{equation}
for any $i,j=1, \ldots, m$ and $h,k=1, \ldots,d$. Under these assumptions, we can prove that conditions \eqref{realLegendrefunc} and \eqref{comLeg} are both satisfied with $\mathscr{C}=mdk_0 $. Indeed by \eqref{cond_i} we get immediately that
\begin{align*} 
{\rm Re}\sum_{h,k=1}^d(A^{hk}(x)\vartheta^k,\vartheta^h)\ge 0, \qquad\,\, \vartheta^k, \vartheta^h \in \mathbb{C}^m,
\end{align*}
for every $x\in\Rd$,
proving the first inequality in \eqref{realLegendrefunc}. 
Further, using \eqref{cond_ii} we can estimate
\begin{align}\label{real_1}
{\rm Re}\sum_{h,k=1}^d(A^{hk}(x)\vartheta^k,\vartheta^h) 
&\le  \sum_{h, k=1}^d\sum_{i,j=1}^m  |a^{hk}_{ij}(x)||\vartheta^k_j| |\vartheta^h_i|\notag\\
& \le k_0\lambda_Q(x)\bigg(\sum_{k=1}^d\sum_{i=1}^m|\vartheta^k_i|\bigg)^2\notag\\
& \le md k_0  \lambda_Q(x)\sum_{k=1}^d\sum_{i=1}^m|\vartheta^k_i|^2\notag\\
& \le mdk_0 {\rm Re}\sum_{i=1}^m\sum_{h,k=1}^d(q_{hk}(x)\vartheta_i^k, \vartheta_i^h)
\end{align}
for any $\vartheta^k=(\vartheta^k_1,\ldots, \vartheta^k_m), \vartheta^h=(\vartheta^h_1,\ldots, \vartheta^h_m) \in \Cm$ and $x \in \Rd$, and the second part of \eqref{realLegendrefunc} follows.
using again \eqref{cond_ii}, we can estimate
\begin{align}
&\bigg |{\rm  Im}\sum_{h,k=1}^d(A^{hk}(x) \vartheta^k,\vartheta^h)\bigg |\notag\\
\le &   
 \sum_{h, k=1}^d\sum_{i,j=1}^m |a^{hk}_{ij}|\big(|{\rm Im}\,\vartheta^k_j||{\rm Re}\,\vartheta_i^h|+ |{\rm Re}\,\vartheta^k_j|{\rm Im}\,\vartheta^h_i|
\big)\notag\\
\le &\frac{1}{2}  \sum_{h,k=1}^d\sum_{i,j=1}^m |a^{hk}_{ij}|\big(|{\rm Im}\,\vartheta^k_j|^2+|{\rm Re}\,\vartheta_i^h|^2+ |{\rm Re}\,\vartheta^k_j|^2+{\rm Im}\vartheta^h_i|^2
\big)\notag\\
=& mdk_0 \lambda_Q(x)\sum_{k=1}^d\sum_{i=1}^m[({\rm Re}\,\vartheta^k_i)^2+({\rm Im}\,\vartheta^k_i)^2]\notag\\
\le &mdk_0\sum_{i=1}^m\sum_{h,k=1}^dq_{hk}[{\rm Re}\,\vartheta^k_i{\rm Re}\,\vartheta^h_i+{\rm Im}\,\vartheta^k_i{\rm Im}\,\vartheta^h_i]\notag\\
= &md k_0 {\rm Re}\,\sum_{i=1}^m(Q(x)\vartheta_i, \vartheta_i),
\label{im_1}
\end{align}
whence also condition \eqref{comLeg} is satisfied with $\mathscr{C}=mdk_0$. Further, if $2mdk_0<1$, then also the results in Section 4 apply.

\subsection{The symmetric case II}
Here, we assume again that $Q$ is as described in (1) and that the $m\times m$ matrices $A^{hk}$ have the form 
\begin{align*}
A^{hk}=q_{hk}G,\qquad\;\, h,k=1, \ldots,d,
\end{align*}
where $G$ is a $m\times m$ matrix-valued function having nonnegative and bounded entries $g_{ij}:\Rd \to \R$  and
$(G(x)\xi, \xi)\ge 0$ for any $x \in \Rd$ and $\xi \in \R^m.$
In this case there exists a positive constant $\Lambda_G$ such that $(G(x)\xi,\xi)\le \Lambda_G|\xi|^2$ for any $\xi \in \R^m$.
In order to check conditions \eqref{realLegendrefunc} and \eqref{comLeg}, we observe that 
\begin{eqnarray*}
\sum_{h,k=1}^d(A^{hk}\theta^k,\theta^h) = \sum_{h,k=1}^d\sum_{i,j=1}^m q_{hk}g_{ij}\theta^k_i\overline\theta^h_j= \sum_{i,j=1}^m g_{ij} (Q \theta_i,\theta_j)
\end{eqnarray*}
for any $\theta^1,\ldots,\theta^d\in\Cm$,
where by $\theta_i$ we denote the vector in $\mathbb C^d$ having coordinates $(\theta_i^1, \ldots, \theta_i^d)$ for any $i \in \{1, \ldots,m\}$. Thus, we obtain that
\begin{align*}
\re \sum_{h,k=1}^d(A^{hk}\theta^k,\theta^h)
=&\sum_{i,j=1}^mg_{ij}\big((Q^{1/2}\re \theta_i, Q^{1/2}\re \theta_j)+(Q^{1/2}\im \theta_i, Q^{1/2}\im \theta_j)\big ),
\end{align*}
which is nonnegative by the assumption on the matrix-valued function $G$. Moreover, since for any $\xi \in \R^m$
\begin{align*}
\sum_{i,j=1}^m g_{ij}(Q^{1/2}\xi_i, Q^{1/2}\xi_j)
= \sum_{\ell=1}^d (G(x)\eta^{\ell},\eta^{\ell})
\le \Lambda_G  \sum_{\ell=1}^d |\eta^{\ell}|^2,
\end{align*}
where $\eta^{\ell}=(\eta_1^{\ell},\ldots,\eta_m^{\ell})$ and $\eta_i^{\ell}=(Q^{1/2}\xi_i)_{\ell}$ for any $i=1, \ldots, m$ and ${\ell}=1, \ldots, d$, we conclude that
\begin{align*}
\re\!\!\sum_{h,k=1}^d(A^{hk}\theta^k,\theta^h)
&\le \Lambda_G  \bigg (\sum_{\ell=1}^d\sum_{i=1}^m (Q^{1/2}\re \theta_i)_{\ell}^2+(Q^{1/2}\im \theta_i)_{\ell}^2\bigg )\\
&=\Lambda_G \re\sum_{i=1}^m\sum_{h,k=1}^d q_{hk}\theta^k_i\overline\theta^h_i
\end{align*}
Analogously, taking into account that $g_{ij}$ are nonnegative functions, we can estimate
\begin{align*}
\left|\im \sum_{h,k=1}^d(A^{hk}\theta^k,\theta^h)\right|
&\leq \sum_{i,j=1}^m g_{ij}|\im (Q\theta_i, \theta_j)|\\
& \le \sum_{i,j=1}^m g_{ij}|Q^{1/2}\theta_i||Q^{1/2}\theta_j|\\
& \le \Lambda_G \sum_{i=1}^m|Q^{1/2}\theta_i|^2\\
&= \Lambda_G \sum_{i=1}^m (Q\theta_i,\theta_i),
\end{align*}
whence conditions \eqref{realLegendrefunc} and \eqref{comLeg} are both satisfied with $\mathscr{C}=\Lambda_G$.

Finally, we consider the case when $Q$ is as described in $(2)$, i.e., $Q$ is a diagonal perturbation of an antysimmetric
matrix. 
We show how the form of the matrix $Q$, in this case, allows us to require slightly weaker assumptions on the entries of the matrices $A^{hh}$.

\subsection{Diagonal perturbation of the antysimmetric case}
Here, we suppose that $Q$ is as described in (2) and, concerning the entries of the matrices $A^{hk}$, besides \eqref{cond_i}, we assume the following conditions:
\begin{enumerate}[\rm (i)]
    \item $0 \le a^{hh}_{ij}\le k_2 q_{hh}$ for any $i,j=1, \ldots, m$, $h=1, \ldots,d$ and some positive constant $k_2$;
\item there exists a positive constant $k_3$ such that
$|a^{hk}_{ij}|\le k_3 \min_{r=1,\ldots, m}q_{rr}$ for any $i,j=1\ldots, m$ and $h,k=1,\ldots,d$ with $h \neq k$.
\end{enumerate}
In this case, arguing as in \eqref{real_1} for the terms $(A^{hk}\vartheta^k, \vartheta^h)$ when $h \neq k$, we can estimate
\begin{align*}
\sum_{h,k=1}^d\sum_{i,j=1}^m a^{hk}_{ij}(x)\zeta^k_j \zeta^h_i&= \sum_{h=1}^d\sum_{i,j=1}^m a^{hh}_{ij}(x)\zeta^h_j \zeta^h_i+ \sum_{h\neq k}\sum_{i,j=1}^m a^{hk}_{ij}(x)\zeta^k_j \zeta^h_i\\
&\le \frac{1}{2}\sum_{h=1}^d\sum_{i,j=1}^m a^{hh}_{ij}[(\zeta^h_j)^2+(\zeta^h_i)^2]+mdk_3\sum_{i=1}^m\sum_{h=1}^dq_{hh}(\zeta_i^h)^2\\
& \le mk_2\sum_{h=1}^d\sum_{i=1}^m q_{hh}(\zeta^h_i)^2+ mdk_3\sum_{i=1}^m q_{hh}(\zeta^h_i)^2\\
&\le (k_2+dk_3)\sum_{i=1}^m\sum_{h,k=1}^dq_{hk}\zeta_i^k\zeta_i^h
\end{align*}
for any $\zeta^k=(\zeta^k_1,\ldots, \vartheta^k_m), \zeta^h=(\zeta^h_1,\ldots, \zeta^h_m) \in \R^m$, where in the last equality we used \eqref{real_part}. Applying this chain of inequalities first with $\zeta^h={\rm Re}\,\vartheta^h$ ($h=1,\ldots,d$) and then with 
$\zeta^h={\rm Im}\,\vartheta^h$ ($h=1,\ldots,d$) and
$\vartheta^1,\ldots,\vartheta^d\in\mathbb C^m$, we conclude that
\begin{align*}
{\rm Re}\sum_{h,k=1}^d(A^{hk}\vartheta^k,\vartheta^h)
\le m(k_2+dk_3)\sum_{i=1}^m\sum_{h,k=1}^dq_{hk}\vartheta_i^k\overline\vartheta_i^h
\end{align*}
for every $\vartheta^1,\ldots,\vartheta^d\in\mathbb C^m$.

Analogously, repeating the same arguments in the proof of \eqref{im_1}, for the terms $(A^{hk}\vartheta^k, \vartheta^h)$ with $h \neq k$, we get
\begin{align*}
\bigg |{\rm  Im}\sum_{h,k=1}^d(A^{hk} \vartheta^k,\vartheta^h)\bigg |
\le &\sum_{h=1}^d\sum_{i,j=1}^m a^{hh}_{ij}\left|\big({\rm Im}\,\vartheta^h_j{\rm Re}\,\vartheta_i^h- {\rm Re}\,\vartheta^h_j{\rm Im}\,\vartheta^h_i
\big)\right|\\
&+mdk_3\sum_{i=1}^m\sum_{h=1}^dq_{hh}(\vartheta_i^h)^2\\
\le &\frac{1}{2}\sum_{h=1}^d\sum_{i,j=1}^m a^{hh}_{ij}(|\vartheta^h_j|^2+|\vartheta_i^h|^2)+mdk_3\sum_{i=1}^m\sum_{h=1}^dq_{hh}(\vartheta_i^h)^2\\
\le &\frac{k_2}{2}\sum_{h=1}^d q_{hh}\sum_{i,j=1}^m (|\vartheta^h_j|^2+|\vartheta_i^h|^2)+mdk_3\sum_{i=1}^m\sum_{h=1}^dq_{hh}(\vartheta_i^h)^2\\
= &m(k_2+dk_3){\rm Re}\sum_{h,k=1}^d\sum_{i=1}^m q_{hk}\vartheta^k_i \overline{\vartheta^h_i},
\end{align*}
where, in the last equality we used \eqref{real_part}. Thus, conditions \eqref{realLegendrefunc} and \eqref{comLeg} are both satisfied with $\mathscr{C}=md(k_2+k_3)$. Further, if $md(k_2+k_3)<1/2$, then also the results in Section 4 apply.

\appendix

\section{Auxiliary results from Linear Algebra}

For every $h,k=1,\ldots, d$, let $A^{hk}$ be a $m\times m$ real-valued matrix, with entries $a^{hk}_{ij}$, such that
\begin{equation}
a^{hk}_{ij}=a^{kh}_{ji}, \qquad\;\, h,k=1, \dots, d,\;\,i,j=1, \dots m. 
\label{sym} 
\end{equation}
Then, $\sum_{h,k=1}^d (A^{hk}\theta^k, \theta^h)\in\R$
for every $\theta^1,\ldots\theta^d \in \mathbb C^m$.

Further, if we assume the positivity condition 
$\sum_{h,k=1}^d (A^{hk}\theta^k, \theta^h)\geq 0$ for every $\theta^1,\ldots,\theta^d\in \Cm$,
then the Cauchy Schwarz inequality
\begin{equation}\label{CS}\bigg |\sum_{h,k=1}^d (A^{hk}\theta^k, \eta^h)\bigg | \leq
\bigg(\sum_{h,k=1}^d (A^{hk}\theta^k, \theta^h)\bigg)^{\frac 1 2} \bigg(\sum_{h,k=1}^d (A^{hk}\eta^k, \eta^h)\bigg)^{\frac 1 2} \end{equation}
holds true for every $\theta^1,\ldots,\theta^d\in \Cm$.

On the other hand, if the matrices $A^{hk}$ satisfy the condition
\begin{equation}
a^{hk}_{ij}=-a^{kh}_{ji}, \qquad\;\, h,k=1, \dots, d,\;\,
 i,j=1, \dots m, 
\label{antisym} 
\end{equation}
then $\displaystyle\sum_{h,k=1}^d (A^{hk}\theta^k, \theta^h)\in i\R$.

More generally,  every real matrix $A^{hk}$ can be split  into the sum $A^{hk}=A^{hk}_s+A^{hk}_{as}$, where
\begin{eqnarray*}
A_s^{hk}=\frac{A^{hk}+(A^{kh})^t}{2} , \qquad A_{as}^{hk}=
\frac{A^{hk}-(A^{kh})^t}{2}.
\end{eqnarray*}
It is clear that $A^{hk}_s$ satisfies condition \eqref{sym} and $(A_{as}^{hk})_{h,k=1, \dots, d}$ satisfies condition \eqref{antisym}.
Moreover, 
  \begin{align}
  &{\rm Re} \sum_{h,k=1}^d (A^{hk}\theta^k, \theta^h) = \sum_{h,k=1}^d (A_s^{hk}\theta^k, \theta^h),\notag\\
  &{\rm Im} \sum_{h,k=1}^d (A^{hk}\theta^k, \theta^h) =\frac{1}{i}\sum_{h,k=1}^d (A_{as}^{hk}\theta^k, \theta^h)
  \label{imforma} 
  \end{align}
for every $\theta^1,\ldots,\theta^d\in\Cm$.

Next proposition yields a sort of mixed Cauchy-Schwarz inequality which is crucial in this paper.
Although it is based on  essentially known techniques, by sake of completeness,  we provide a sketch of the proof.

\begin{prop}\label{CSconf} 
For every $h,k=1.\ldots$, let $A^{hk}$  and $M^{hk}$ be two $m\times m$  real-valued matrices such that
 \begin{align}\label{RA}
& \bigg |{\rm Im}\sum_{h,k=1}^d (A^{hk}\theta^k, \theta^h)\bigg |\leq C_0 {\rm Re}\sum_{h,k=1}^d (M^{hk}\theta^k, \theta^h)
  \end{align}
for some positive constant $C_0$ and every $\theta^1,\ldots,\theta^d\in \Cm$. Then, the inequality
  \begin{equation}
  \label{modifCS1}
  \bigg |\sum_{h,k=1}^d (A_{as}^{hk}\theta^k, \eta^h)\bigg |\leq C_0
\bigg( {\rm Re}\sum_{h,k=1}^d (M^{hk}\theta^k, \theta^h)\bigg)^{\frac 1 2} \bigg( {\rm Re}\sum_{h,k=1}^d (M^{hk}\eta^k, \eta^h)\bigg)^{\frac 1 2} \end{equation}
holds true for every $\theta^1,\ldots,\theta^d$ and $\eta^1,\ldots,\eta^d$ in $\Cm$.

If, in addition, 
\begin{equation}
0\leq {\rm Re}\sum_{h,k=1}^d (A^{hk}\theta^k, \theta^h)\leq C_1{\rm Re}\sum_{h,k=1}^d (M^{hk}\theta^k, \theta^h)
\label{ondulato}
\end{equation}
for some positive constant $C_1$ and every $\theta^1,\ldots.\theta^d\in \Cm$,
then \eqref{modifCS1} holds true with the matrices $A^{hk}_{as}$ being replaced by $A^{hk}$ and with the constant $C_0$ being replaced by the constant $C_0+C_1$.
\end{prop}

\begin{proof}
Assume first  that  $\sum_{h,k=1}^d (A_{as}^{hk}\eta^k, \theta^h)\in i\R$.  Then, by \eqref{imforma} and observing that
$\sum_{h,k=1}^d(A^{hk}_{as}\eta^k,\theta^h)=-\overline{\sum_{h,k=1}^d(A^{hk}_{as}\theta^k,\eta^h)}$, it follows that
  \begin{align*}
  &\frac{i}{4} \bigg ( {\rm Im}\sum_{h,k=1}^d (A^{hk}(\eta^k + \theta^k), \eta^h +\theta^h) -
   {\rm Im}\sum_{h,k=1}^d (A^{hk}(\eta^k - \theta^k), \eta^h -\theta^h)\bigg )\\
 &\frac 1 4 \bigg ( \sum_{h,k=1}^d (A_{as}^{hk}(\eta^k + \theta^k), \eta^h +\theta^h) -
   \sum_{h,k=1}^d (A_{as}^{hk}(\eta^k - \theta^k), \eta^h -\theta^h)\bigg )\\
 &= \frac 1 2 \bigg ( \sum_{h,k=1}^d (A_{as}^{hk}\theta^k, \eta^h) -\overline{ \sum_{h,k=1}^d (A_{as}^{hk}\theta^k, \eta^h)} \bigg )=\sum_{h,k=1}^d (A_{as}^{hk}\theta^k, \eta^h).
 \end{align*}
 Hence, using \eqref{RA} we obtain
 \begin{align}
 &\bigg |\sum_{h,k=1}^d (A_{as}^{hk}\theta^k, \eta^h)\bigg |\notag\\
\leq & \frac{C_0}{4} \bigg ( \sum_{h,k=1}^d {\rm Re} \left(M^{hk}(\eta^k + \theta^k), \eta^h +\theta^h\right)+\sum_{h,k=1}^d{\rm Re} \left(M^{hk}(\eta^k - \theta^k), \eta^h -\theta^h\right)\bigg )\nonumber\\
=& \frac{C_0}{2} {\rm Re}  \sum_{h,k=1}^d (M^{hk}\eta^k, \eta^h) +  \frac{C_0}{2} {\rm Re}  \sum_{h,k=1}^d (M^{hk}\theta^k, \theta^h).
\label{ho} 
\end{align}

Writing \eqref{ho} with $\theta^k$ and $\eta^h$ being replaced, respectively, by to $\sqrt{\varepsilon}\theta^k$ and $\sqrt{\varepsilon}\eta^h$, we get
\begin{equation*}
\bigg |\sum_{h,k=1}^d (A_{as}^{hk}\theta^k, \eta^h)\bigg | \leq
   \frac{c_0}{2 \sqrt{\varepsilon} }{\rm Re} \sum_{h,k=1}^d (M^{hk}\eta^k, \eta^h) + \sqrt{\varepsilon }\frac{c_0}{2} {\rm Re}  \sum_{h,k=1}^d (M^{hk}\theta^k, \theta^h)
   \end{equation*}
for every $\varepsilon>0$. By taking the minimum over $\varepsilon>0$, estimate \eqref{modifCS1} follows in this particular case.

To get estimate \eqref{modifCS1} in the general case, we assume that $\sum_{h,k=1}^d (A_{as}^{hk}\theta^k,\eta^h)$ does not belong to $i\mathbb R$ and write  $\sum_{h,k=1}^d (A_{as}^{hk}\theta^k,\eta^h)=r e^{i\psi}$ for some $\psi\in\R$ and $r\ge 0$. Since $\sum_{h,k=1}^d(A_{as}^{hk}\theta^k,ie^{i\psi} \eta^h)\in i\R$,
applying \eqref{modifCS1} to $\theta^1,\ldots,\theta^d$ and $ie^{i\psi}\eta^1,\ldots,ie^{i\psi}\eta^d$, estimate \eqref{modifCS1} follows in its full generality.

Finally, we assume that condition \eqref{ondulato} is satisfied.
Then, combining \eqref{CS} and \eqref{ondulato}, we can estimate
\begin{align}
\bigg |\sum_{h,k=1}^d (A_s^{hk}\theta^k, \eta^h)\bigg |
\leq C_1\bigg( {\rm Re}\sum_{h,k=1}^d (M^{hk}\theta^k, \theta^h)\bigg)^{\frac{1}{2}}\bigg( {\rm Re}\sum_{h,k=1}^d (M^{hk}\eta^k, \eta^h)\bigg)^{\frac{1}{2}}.
\label{form-14}
\end{align}
Hence, from \eqref{modifCS1} and \eqref{form-14} and recalling that $A^{hk}=A^{hk}_s+A^{hk}_{as}$ for every $h,k=1,\ldots,d$, we conclude that
\begin{align*}
\bigg |\sum_{h,k=1}^d (A^{hk}\theta^k, \eta^h)\bigg | \le &(C_0+C_1) \bigg({\rm Re}\sum_{h,k=1}^d (M^{hk}\theta^k, \theta^h)\bigg)^{\frac{1}{2}}\bigg ({\rm Re}\sum_{h,k=1}^d (M^{hk}\eta^k, \eta^h)\bigg)^{\frac{1}{2}}.
\end{align*}
The proof is complete.
\end{proof}

 By applying Proposition \ref{CSconf} with $M^{hk}=A^{hk}$, we get the following result.

\begin{coro}
Suppose that there exists a positive constant $C_0>0$ such that
 \begin{equation*}  
 \bigg | {\rm Im}\sum_{h,k=1}^d (A^{hk}\theta^k, \theta^h)\bigg |\leq C_0{\rm Re}\sum_{h,k=1}^d (A^{hk}\theta^k, \theta^h)
  \end{equation*}
for every $\theta^1,\ldots,\theta^d\in \Cm$. Then,
  \begin{eqnarray*}
  \label{modifCS}\bigg |\sum_{h,k=1}^d (A^{hk}\theta^k, \eta^h)\bigg | \leq (1+C_0)
\bigg( {\rm Re}\sum_{h,k=1}^d (A^{hk}\theta^k, \theta^h)\bigg)^{\frac 1 2} \bigg( {\rm Re}\sum_{h,k=1}^d (A^{hk}\eta^k, \eta^h)\bigg)^{\frac 1 2}
\end{eqnarray*}
for every $\theta^1,\ldots,\theta^d, \eta^1,\ldots, \eta^d\in\Cm$.
\end{coro}

\begin{rmk} 
{\rm If $(A^{hk})^T=A^{kh}$ for every $h,k=1,\ldots,d$, then the imaginary part of $\sum_{h,k=1}^d (A^{hk}\theta^k, \theta^h)$ is zero, by  \eqref{imforma}. So, if estimate \eqref{realLegendrefunc} holds true, then condition \eqref{comLeg} is satisfied if the real part of $\sum_{h,k=1}^d (A^{hk}\theta^k, \theta^h)$ is nonnegative
for every $\theta^1,\ldots,\theta^d\in\Cm$.}
\end{rmk}

\section{Derivation of $\Lambda_p$}
\label{app-B}
We need to compute the supremum of $\psi_2$ defined in \eqref{f2} on the set 
\begin{eqnarray*}
\Omega_p=\{(\varepsilon_0, \varepsilon_1,\varepsilon_2, \varepsilon_3)\in (0,+\infty)^4: \tilde \psi_1(\varepsilon_0, \varepsilon_1, \varepsilon_2, \varepsilon_3)\ge 0\},
\end{eqnarray*}
where 
$\tilde \psi_1=\psi_1$ if $p\ge 2$ and
$\tilde \psi_1=\psi_1+p-2$ if $p\in (1,2)$.
Since $\psi_2$ is continuous on $\Omega$, the supremum of $\psi_2$ on $\Omega_p$ is achieved on $\partial \Om$. To simplify the notation, we set
\begin{eqnarray*}
\tilde\psi_1= e_1-a_1 \varepsilon_0-b_1\varepsilon_1-c_1\varepsilon_2-d_1\varepsilon_3
\end{eqnarray*}
and
\begin{eqnarray*}
\psi_2=1-\frac{a_2}{\varepsilon_0}-\frac{b_2}{\varepsilon_1}-\frac{c_2}{\varepsilon_2}-\frac{d_2}{\varepsilon_3},
\end{eqnarray*}
where 
\begin{eqnarray*}
\begin{array}{ll}
a_1=4a_2=(p-1)(1+c_0)\gamma,\quad &b_1=4b_2v_0^3=(p-1)(1+c_0)C_\gamma,\\[1mm]
c_2=c_1=\mathscr{C}(p-1)\gamma, &d_1=v_0^3d_2=\mathscr{C}(p-1)C_\gamma,
\end{array}
\end{eqnarray*}
\begin{eqnarray*}
e_1=\left\{
\begin{array}{ll}
1-2\mathscr{C}(2p-3), \quad\quad &p \ge 2\\[1mm]
p-1-2\mathscr{C}(5-2p),\quad \quad &p\in (1,2),
\end{array}
\right.
\end{eqnarray*}
\begin{eqnarray*}
a_2=\frac{(p-1)(1+c_0)\gamma}{4},\quad\quad b_2=\frac{(p-1)(1+c_0)C_\gamma}{4v_0^3},
\end{eqnarray*}
\begin{eqnarray*}
c_2=c_1=\mathscr{C}(p-1)\gamma,\qquad d_2=\frac{\mathscr{C}(p-1)C_\gamma}{v_0^3}.
\end{eqnarray*}
Since the function $\psi_2$ does not admit stationary points in the interior of $\Omega$ and it is continuous over the closed and bounded set $\Omega$, it achieves its maximum value at the boundary of $\Omega$. Hence, we can assume that
$\varepsilon_0=a_1^{-1}(e_1-b_1\varepsilon_1-c_1\varepsilon_2-d_1\varepsilon_3)$ and compute the supremum of the function
\begin{eqnarray*}
\tilde\psi_2(\varepsilon_1,\varepsilon_2, \varepsilon_3)=1-\frac{a_1a_2}{e_1-b_1\varepsilon_1-c_1\varepsilon_2-d_1\varepsilon_3}-\frac{b_2}{\varepsilon_1}-\frac{c_2}{\varepsilon_2}-\frac{d_2}{\varepsilon_3}
\end{eqnarray*}
on the set 
\begin{eqnarray*}
\Gamma=\{(\varepsilon_1,\varepsilon_2, \varepsilon_3)\in (0,+\infty)^3: e_1-b_1\varepsilon_1-c_1\varepsilon_2-d_1\varepsilon_3>0\}.
\end{eqnarray*}
The coefficient $e_1$ is strictly positive due to condition $p\in \left (1+\frac{6\mathscr{C}}{4{\mathscr C}+1},\frac{3}{2}+\frac{1}{4\mathscr C}\right )$ and  allows us to consider simultaneously  the cases $p \ge 2$ and $p \in (1,2)$.
Since $\tilde\psi_2$ diverges to $-\infty$ on $\partial \Gamma$, it achieves its supremum on $\Gamma$. To find it, we solve the system $\nabla \tilde f_2=(0,0,0)$, i.e.,
\begin{equation*}\left\{
\begin{array}{ll}
b_2(e_1-b_1\varepsilon_1-c_1\varepsilon_2-d_1\varepsilon_3)^2=\varepsilon_1^2 a_1a_2b_1,\\[1mm]
c_2(e_1-b_1\varepsilon_1-c_1\varepsilon_2-d_1\varepsilon_3)^2=\varepsilon_2^2 a_1a_2c_1,\\[1mm]
d_2(e_1-b_1\varepsilon_1-c_1\varepsilon_2-d_1\varepsilon_3)^2=\varepsilon_3^2 a_1a_2d_1,
\end{array}
\right.
\end{equation*}
whence 
\begin{equation*}\left\{
\begin{array}{ll}
(\sqrt{a_1a_2}+\sqrt{b_1b_2})\varepsilon_1+c_1\sqrt{\frac{b_2}{b_1}}\varepsilon_2+d_1\sqrt{\frac{b_2}{b_1}}\varepsilon_3=e_1\sqrt{\frac{b_2}{b_1}},\\[1.5mm]
b_1\sqrt{\frac{c_2}{c_1}}\varepsilon_1+ (\sqrt{a_1a_2}+\sqrt{c_1c_2})\varepsilon_2+d_1\sqrt{\frac{c_2}{c_1}}\varepsilon_3=e_1\sqrt{\frac{c_2}{c_1}},\\[1.5mm]
b_1\sqrt{\frac{d_2}{d_1}}\varepsilon_1+c_1\sqrt{\frac{d_2}{d_1}}\varepsilon_2+(\sqrt{a_1a_2}+\sqrt{d_1d_2})\varepsilon_3=e_1\sqrt{\frac{d_2}{d_1}},
\end{array}
\right.
\end{equation*}
which yields
\begin{align*}
&\varepsilon_1=\frac{e_1\sqrt{b_2}}{\sqrt{b_1}(\sqrt{a_1a_2}+\sqrt{b_1b_2}+\sqrt{c_1c_2}+\sqrt{d_1d_2})},\\[1mm]
&\varepsilon_2=\frac{e_1\sqrt{c_2}}{\sqrt{c_1}(\sqrt{a_1a_2}+\sqrt{b_1b_2}+\sqrt{c_1c_2}+\sqrt{d_1d_2})},\\[1mm]
&\varepsilon_3=\frac{e_1\sqrt{d_2}}{\sqrt{d_1}(\sqrt{a_1a_2}+\sqrt{b_1b_2}+\sqrt{c_1c_2}+\sqrt{d_1d_2})}
\end{align*}
and, consequently, 
we infer that the supremum of $\psi_2$ is given by
\begin{align*}
&1-\frac{(\sqrt{a_1a_2}+\sqrt{b_1b_2}+\sqrt{c_1c_2}+\sqrt{d_1d_2})^2}{e_1}\\
=&  \left\{
\begin{array}{ll}
1-\displaystyle\frac{(p-1)^2(\gamma+C_\gamma v_0^{-3/2})^2(1+c_0+2\mathscr{C})^2}{4[p-1-2\mathscr{C}(5-2p)]},\quad p \in (1,2),\\[3mm]
\displaystyle
1-\frac{(p-1)^2(\gamma+C_\gamma v_0^{-3/2})^2(1+c_0+2\mathscr{C})^2}{4[(1-2\mathscr{C}(2p-3)]},\quad p \ge 2.
\end{array}
\right.
\end{align*}


\begin{thebibliography}{99}
\bibitem{AALT}
D. Addona, L. Angiuli, L. Lorenzi, G. Tessitore,
\newblock{\em On coupled systems of Kolmogorov equations with applications to stochastic differential games},
\newblock{ESAIM: Control. Optim. Calc. Var. {\bf 23} (2017), 937-976.}

\bibitem{AAL_Inv}
D. Addona, L. Angiuli, L. Lorenzi,
\newblock{\em On invariant measures associated to weakly coupled systems of Kolmogorov equations},
\newblock{Adv. Differential Equations {\bf 24} (2019), 137-184.}

\bibitem{AAL_Inv1}
D. Addona, L. Angiuli, L. Lorenzi,
\newblock{\em Invariant measures for systems of Kolmogorov equations},
\newblock{J. Appl. Anal. Comput.} {\bf 8} (2018), 764-804.

\bibitem{AL}
D. Addona, L. Lorenzi,
\newblock{\em On weakly coupled systems of partial differential equations with different diffusion terms}, Commun. Pure Appl. Anal, doi/10.3934/cpaa.2022154.

\bibitem{ALM} A. A. Albanese, L. Lorenzi, E. Mangino,
\newblock{\em $L^p$-uniqueness for elliptic operators with unbounded coefficients in $\R^N$ },
\newblock{J. Funct. Anal.} {\bf  256} (2009), 1238-1257.

\bibitem{AL}
L. Angiuli, L. Lorenzi,
\newblock{\em On coupled systems of PDEs with unbounded coefficients},
\newblock{Dyn. Partial Differ. Equ.} \textbf{17} (2020), 129-163.

\bibitem{AngLorMan}
L. Angiuli, L. Lorenzi, E.M. Mangino,
\newblock{\em On a perturbation of a class of Schr\"odinger systems in $L^2$ spaces},
\newblock{Note Mat.} {\bf 38 } (2018), 125-138.

\bibitem{ALMR}
L. Angiuli, L. Lorenzi, E. M. Mangino, A. Rhandi,
\newblock{\em Generation results for vector-valued elliptic operators with unbounded coefficients in $L^p$ spaces},
\newblock{Ann. Mat. Pura Appl.} {\bf 201(3)}  (2022), 1347-1379.


\bibitem{ALMR1}
L. Angiuli, L. Lorenzi, E. M. Mangino, A. Rhandi,
\newblock{\em On vector-valued Schr\"odinger operators with unbounded diffusion in $L^p$ spaces}, \newblock{J. Evol. Equ.} {\bf 21(3)} (2021), 3181-3204.
 

\bibitem{AngLorPal}
L. Angiuli, L. Lorenzi, D. Pallara,
\newblock{\em $L^p$ estimates for parabolic systems with unbounded coefficients coupled at zero and first order},
\newblock{ J. Math. Anal. Appl.} {\bf 444} (2016), 110-135.


\bibitem{Au} P. Auscher, \newblock{\em On necessary and sufficient conditions for $L^p$-estimates of Riesz transforms associated to elliptic operators on $\R^n$ and related estimates},
\newblock{Mem. Amer. Math. Soc.} {\bf 186} (2007), no. 871.

\bibitem{CD} A. Carbonaro, O. Dragicevic, \newblock{\em Convexity of power functions and bilinear embedding for divergence-form operators with complex coefficients}, \newblock{J. Eur. Math. Soc.} {\bf 22(10)} (2020), 3175-3221.

\bibitem{CM}
A. Cialdea, V. Maz'ya, \newblock{ Semi-bounded Differential Operators, Contractive Semigroups and Beyond},\newblock{Operator Theory: Advances and Applications, \textbf{243}, Birkh\"auser, Berlin, 2014}.




\bibitem{davies_1995} E.B. Davies, \newblock{\em Uniformly elliptic operators with measurable coefficients}, \newblock{ J. Funct. Anal.} {\bf 132 (1)} (1995), 141-169.

\bibitem{DL}
S. Delmonte, L. Lorenzi,
\newblock{\em On a class of weakly coupled systems of elliptic operators with unbounded coefficients.} \newblock{Milan J. Math.} \textbf{79} (2011), 689-727.

\bibitem{egert1}
M. Egert, \newblock{\em $L^p$-estimates for the square root of elliptic systems with mixed boundary conditions}, \newblock{J. Differential Equations} {\bf  265(4)} (2018), 1279-1323.

\bibitem{egert} M. Egert, \newblock{\em On $p$-elliptic divergence form operators and holomorphic semigroups}, \newblock{J. Evol. Equ.} {\bf 20} (2020), 705-724

\bibitem{engnagel}
K.J. Engel, R. Nagel,
\newblock{One-Parameter Semigroups for Linear Evolution Equations}, \newblock{Springer-Verlag, New York, 2000.}

\bibitem{goldstein}
J.A. Goldstein,
\newblock{Semigroups of linear operators and applications}, Oxford University Press, New York, Clarendon Press, Oxford, 1985.

\bibitem{HLPRS}
M. Hieber, L. Lorenzi, J. Pr{\"u}ss, A. Rhandi, R. Schnaubelt,
\newblock{\em Global properties of generalized Ornstein-Uhlenbeck operators on $L^p(\mathbb{R}^N,\mathbb{R}^N)$ with more than linearly growing coefficients},
\newblock{J. Math. Anal. Appl.} {\bf 350} (2009), 100-121.


\bibitem{HMM} S. Hofmann, S. Mayboroda, A. McIntosh, \newblock{\em Second order elliptic operators with complex bounded measurable coefficients in $L^p$, Sobolev and Hardy spaces}, \newblock{Ann. Sci. \'Ec. Norm. Sup\'er.} {\bf 44(5)}  (2011), 723-800.


\bibitem{KM} G. Kresin, V. Maz'ya,
\newblock{Maximum principles and sharp constants for solutions of elliptic and parabolic systems. Mathematical Surveys and Monographs}, \textbf{183}, American Mathematical Society, Providence, RI, 2012.

\bibitem{KLMR}
M. Kunze,  L. Lorenzi, A. Maichine, A. Rhandi,
\newblock{\em $L^p$-theory for Schr\"odinger systems},
\newblock{Math. Nachr. \textbf{292} (2019), 1763-1776}.

\bibitem{KMR}
M. Kunze, A. Maichine, A. Rhandi,
\newblock{\em Vector-valued Schr\"odinger operators on $L^p$ spaces},
\newblock{Discr. Cont. Dyn. Syst. Series S} \textbf{13} (2020), 1529-1541.

\bibitem{newbook}
L. Lorenzi,
\newblock{Analytical Methods for Kolmogorov Equations. Second Edition}, \newblock{Monographs and Research Notes in Mathematics. CRC Press, Boca Raton, FL, 2017.}

\bibitem{MR}
A. Maichine, A. Rhandi,
\newblock{\em On a polynomial scalar perturbation of a Schr\"odinger system in $L^p$-spaces},
\newblock{J. Math. Anal. Appl.} {\bf 466} (2018), 655-675.


\bibitem{miyazaki}
Y. Miyazaki,
\newblock{\em The $L^p$ resolvents for elliptic systems of divergence form}
\newblock{Adv. Stud. Pure Math.}, {\bf 44} (2006), 245-254.

\bibitem{monniaux-pruss}
S. Monniaux, J. Pr\"uss,
\newblock{\em A theorem of the Dore-Venni type for noncommuting operators,}
\newblock{Trans. Amer. Math Soc.} {\bf 349} (1997), 4787-4814.

\bibitem{TerElst} A.F.M. ter Elst, R. Haller-Dintelmann, J. Rehberg, P. Tolksdorf, \newblock{\em On the $L^p$-theory for second-order elliptic operators in divergence form with complex coefficients}, \newblock{J. Evol. Equ.} {\bf 21} (2021), 3963-4003.

\bibitem{Tolk} P. Tolksdorf, \newblock{\em $\mathcal R$-sectoriality of higher-order elliptic systems on general bounded domains}, \newblock{J. Evol.
Equ.}{\bf  18(2)} (2018), 323-349.
\end{thebibliography}
\end{document}